\documentclass[a4paper]{article}

\usepackage[english]{babel}
\usepackage[utf8x]{inputenc}
\usepackage[T1]{fontenc}
\usepackage{dsfont}
\usepackage{amssymb}
\usepackage{mathtools}
\usepackage{amsmath}
\usepackage{amsthm}
\usepackage{xcolor}
\DeclareMathOperator{\PG}{PG}
\DeclareMathOperator{\AG}{AG}
\newcommand{\qbin}[2]{\genfrac{[}{]}{0pt}{}{#1}{#2}}

\usepackage[a4paper,top=3cm,bottom=2cm,left=3cm,right=3cm,marginparwidth=1.75cm]{geometry}

\usepackage{amsmath}
\usepackage{graphicx}
\usepackage[colorinlistoftodos]{todonotes}
\usepackage[colorlinks=true, allcolors=blue]{hyperref}

\newtheorem{definition}{Definition}[section]

\newtheorem{theorem}[definition]{Theorem}

\newtheorem{lemma}[definition]{Lemma}

\newtheorem{remark}[definition]{Remark}
\newtheorem{example}[definition]{Example}

\title{Hilton-Milner results in projective and affine spaces}
\date{}
\author{Jozefien D'haeseleer\thanks{Department of Mathematics: Analysis, Logic and Discrete Mathematics, Ghent University, Krijgslaan 281, Building S8, 9000 Gent, Flanders, Belgium}}
\begin{document}
\maketitle
 
 \begin{abstract}
 	In this article, we analyse  maximal sets of $k$-spaces,  in $\PG(n,q)$ and  $\AG(n,q)$, {$n>2k+t+2$},  that pairwise meet in at least a $t$-space. 
 	It is known that for both $\PG(n,q)$ and $\AG(n,q)$, the largest example is a \emph{$t$-pencil}, i.e. the set of all $k$-spaces containing a fixed $t$-space (see \cite[Theorem 1]{frankl} and \cite[Theorem 1.4]{GuoXu}).  In this paper, we analyse the structure of the second largest maximal example in both $\PG(n,q)$ and $\AG(n,q)$. 
 \end{abstract}

 
\section{Introduction}
\emph{Before
	 we start with the introduction, we would like to indicate how this paper came about. we started investigating the Hilton-Milner problem in the affine context: we studied the second largest examples of sets of affine $k$-spaces pairwise intersecting in at least a $t$-space in $\AG(n,q)$. Thanks to Prof. Tam\'as Sz\H{o}nyi, we received notes of David Ellis about the projective analogue of this problem \cite{ellis}. In these notes, he studied the second largest families of projective $k$-spaces, pairwise intersecting in at least a $t$-space in $\PG(n,q)$. These notes helped me to shorten my, affine,  arguments. Since these notes are not published, we integrate them in this article. The results that are mostly influenced by the ideas in the notes of David Ellis are Lemma \ref{remarks}, Lemma \ref{lemmacounting1}, Lemma \ref{lemmacountingproj} and Lemma \ref{lemma1ellis}. \\
 While finishing the last details of this paper, the paper \cite{nieuw} appeared on Arxiv. In that paper, the authors deduce similar results as ours in the vector space setting. It is worth noting that our results were obtained independently, and our paper deals with both the affine and projective case at once.}\\

A family of subsets that are pairwise intersecting, is called an \textit{Erd{\H{o}}s-Ko-Rado set}. The classification of the largest Erd{\H{o}}s-Ko-Rado sets is called the \textit{Erd{\H{o}}s-Ko-Rado problem}, or \emph{EKR problem}. 
The investigation of EKR problems started in the context of set theory. The first question was  to  determine the size of the largest sets of pairwise non-trivially  intersecting subsets. In 1961,  Erd{\H{o }}s, Ko and Rado solved this problem \cite{erdos_ko_rado}, and their result was improved by Wilson in 1984 \cite{wilsonEKR}.

\begin{theorem}{\textbf{\cite{wilsonEKR}}}\label{thm1}
	Let $n, k$ and $t$ be positive integers and  suppose that $k \geq t \geq 1$ and $n \geq (t+1)(k-t+1)$. If $\mathcal{S}$ is a family of subsets of size $k$ in a set $\Omega$, with $|\Omega | = n$, such that the elements of $\mathcal{S}$ pairwise intersect in at least $t$ elements, then $|\mathcal{S}| \leq \binom{n-t}{k-t}$.\\
	Moreover, if $n \geq (t + 1)(k − t + 1) + 1$, then $|\mathcal{S}| = \binom{n-t}{k-t}$ if and only if $\mathcal{S}$ is the set of all the subsets of size $k$ through a fixed subset  of $\Omega$ of size $t$.
\end{theorem}

    Hilton and Milner \cite{hilton_minler} investigated the largest  Erd{\H{o}}s-Ko-Rado sets $S$ with the property that there is no point contained in all elements of $S$. The classification of the second largest Erd{\H{o}}s-Ko-Rado set is often called a \emph{Hilton-Milner} result.

  The Erd{\H{o}}s-Ko-Rado and Hilton-Milner problem 
  can be translated to many other settings such as projective and affine geometries \cite{surveyEKR}, permutation groups \cite{EKRpermutation} and designs \cite{EKRdesign}. In this article, we work in the projective and affine geometry setting, which is sometimes called a \emph{q-analogue} setting.\\
We present the relevant results in Theorems \ref{frankl} and \ref{q-analogue_Hilton}, but we first briefly recall the definition of $q$\textit{-ary Gaussian coefficient}.
 
 \begin{definition}
 	Let $q$ be a prime power,  let $n,k$ be  non-negative integers with $k \leq n $. The $q$\textit{-ary Gaussian
 		coefficient} of $n$ and $k$ is defined by
 	\begin{equation*}
 	\qbin{n}{k}_q=
 	\begin{cases} 
 	\frac{(q^n-1)\cdots(q^{n-k+1}-1)}{(q^k-1)\cdots(q-1)}  \hspace{1.3cm}\textnormal{ if $k>0$} \\
 	\hspace{1.5cm} 1 \hspace{2.5cm}\textnormal{otherwise} 
 	
 	\end{cases}
 	\end{equation*}
 \end{definition}
 We will write $\qbin{n}{k}$, if the field size $q$ is clear from the context. The number of $k$-spaces in $\PG(n,q)$ is $\qbin{n+1}{k+1}$  and the number of $k$-spaces through a fixed $t$-space  in $\PG(n,q)$, with $0 \leq t \leq k$, is $\qbin{n-t}{k-t}$.
 Moreover, we will denote the number $\qbin{n+1}{1}$ by the symbol $\theta_n$.\\

Frankl and Wilson proved the following  Erd{\H{o}}s-Ko-Rado result in finite projective spaces; they classified the largest example of sets of $k$-spaces pairwise intersecting in at least a  $t$-space.  Here, the set of subspaces  through a fixed $t$-space will be called a \emph{$t$-pencil}, and, in particular, a \textit{point-pencil} if $t=0$.
 \begin{theorem}{\textbf{\cite[Theorem 1]{frankl}}}\label{frankl}
 	Let $t$ and $k$ be integers, with $0 \leq t \leq k$. Let $\mathcal{S}$ be a set of $k$-spaces
 	in $\PG(n,q)$, pairwise intersecting in at least a $t$-space.
 	\begin{itemize}
 		\item [(i)] If $n \geq 2k + 1$, then $| \mathcal{S} |
 		\leq \qbin{n-t}{k-t}$. Equality holds if and only if $\mathcal{S}$ is the set of all the $k$-spaces, containing a fixed $t$-space of $\PG(n,q)$, or $n = 2k +1$ and $\mathcal{S}$ is the set of all the $k$-spaces in a fixed $(2k -t)$-space.
 		\item[(ii)] If $2k-t \leq n \leq 2k$, then $|\mathcal{S}| \leq \qbin{2k-t+1}{k+1}$. Equality holds if and only if $\mathcal{S}$ is the set of all the $k$-spaces in a fixed $(2k-t)$-space.
 	\end{itemize}
 \end{theorem}
  
 Blokhuis \textit{et al}, investigated the Hilton-Milner problem, and classified the second-largest maximal Erd{\H{o}}s-Ko-Rado sets of subspaces in a finite projective space.
  
  \begin{theorem}{\textbf{\cite[Theorem 1.3, Proposition 3.4]{q-analogue_hilton}}}\label{q-analogue_Hilton}
  	Let  $\mathcal{S}$ be a maximal set of pairwise intersecting $k$-spaces in $\PG(n,q)$, with $n \geq 2k + 2$, $k \geq 2$ and $q \geq 3$ (or
  	$n \geq 2k + 4$, $k \geq 2$ and $q = 2$). If $\mathcal{S}$ is not a point-pencil, then
  	\[ |\mathcal{S} |\leq \qbin{n}{k}-q^{k(k+1)}\qbin{n-k-1}{k}+q^{k+1}.\]
  	Moreover, if equality holds, then
  	\begin{itemize}
  		\item [(i)] either $\mathcal{S}$ consists of all the $k$-spaces through a fixed point
  		P, meeting a fixed $(k+1)$-space $\sigma$, with $P \in \sigma$, in a $j$-space, $j \geq 1$, together with all the $k$-spaces in $\sigma$;
  		\item [(ii)]or else $k = 2$ and $\mathcal{S}$ is the set of all the planes meeting a fixed plane $\pi$ in at
  		least a line.
  	\end{itemize}
  \end{theorem}

  In \cite{GuoXu}, Guo and Xu investigated the Erd{\H{o}}s-Ko-Rado problem in affine spaces. They proved that the largest $t$-intersecting family of $k$-spaces in $\AG(n,q)$, $n\geq 2k+t+2$, is the set of all $k$-spaces through a fixed $t$-space.  In Section \ref{sectionappendix}, we give a shorter proof for their result and improve their bound on $n$ to $n\geq 2k+1$.  For $t=0$, the second largest $t$-intersecting set of $k$-spaces in $\AG(n,q)$ was already described in \cite{Gong}. The main goal in this article is to describe the second largest Erd\H{o}s-Ko-Rado sets for $t>0$, for both $\PG(n,q)$ and $\AG(n,q)$.\\
  
  In many cases, we will use counting arguments to find the classification results. For this, we will often use the following lemma.
 \begin{lemma}[{\cite[Section 170]{Segre}}]\label{lemmadisjunct}
	The number of $j$-spaces disjoint from a fixed $m$-space in $\PG(n,q)$ equals $q^{(m+1)(j+1)}\qbin{n-m}{j+1}$.
\end{lemma}
  
  In Section \ref{sectionexampleproj} and in Section \ref{sectionexampleaffien}, we give two examples of maximal sets of $k$-spaces in $\PG(n,q)$ and $\AG(n,q)$ respectively, pairwise intersecting in at least a $t$-space, which are not $t$-pencils. In Section \ref{sectionclassi}, we prove,  the Hilton-Milner results:  
    for $k>2t+2$, Example \ref{examplep} and for $k\leq2t+2$, Example \ref{example2p} is the second largest example of $k$-spaces in $\PG(n,q)$, pairwise meeting in at least a $t$-space.
  For $k>2t+1$, Example \ref{example} and for $k\leq2t+1$, Example \ref{example2} is the second largest example of $k$-spaces in $\AG(n,q)$, pairwise meeting in at least a $t$-space. In both cases we suppose that $q\geq 4$ and $n>2k+t+2$, or $q=3$ and $n>2k+t+3$.

\section{Two  examples in $\PG(n,q)$}\label{sectionexampleproj}
We start by giving two examples of maximal sets of $k$-spaces in PG$(n,q)$, pairwise meeting in at least a $t$-space. 
 Note that for $n\leq 2k-t$,  all projective $k$-spaces in $\PG(n,q)$ are pairwise intersecting in at least a $t$-space. Hence, we may suppose that $n>2k-t$.

\begin{example}\label{examplep}
	Let $\delta$ be a $t$-space, $t<k$, in $\PG(n,q)$, $n>2k-t$, and let $\pi$ be a $k$-space in $\PG(n,q)$ with $\dim(\pi\cap \delta)=t-1$. 
	Let $S_1$ be the set of all  $k$-spaces in $\langle \pi, \delta \rangle$. Let 
	$S_2$ be the set of all $k$-spaces through $\delta$ and meeting $\langle \pi, \delta \rangle$ in at least a $(t+1)$-space.
	Let $\mathcal{S}$ be the union of the sets $S_1$ and $S_2$. 
\end{example}

\begin{lemma}
	The set $\mathcal{S}$, described in Example \ref{examplep}, is a {maximal} set of $k$-spaces in $\PG(n,q)$, $n>2k-t$, pairwise intersecting in at least a $t$-space, of size
	\begin{align*}
	|\mathcal{S}|
	&=\theta_{k+1}-\theta_{k-t}+\qbin{n-t}{k-t}- q^{(k-t+1)(k-t)}\qbin{n-k-1}{k-t}\\
	&=\theta_{k+1}+\sum_{j=0}^{k-t-2} \qbin{k-t+1}{j+1}q^{(k-t-j)(k-t-j-1)}\qbin{n-k-1}{k-t-j-1}.
	\end{align*}
\end{lemma}

\begin{proof}
We start with determining the size of $\mathcal{S}$. Note first that the number of elements of $S_1\setminus S_2$ is equal to the number of $k$-spaces in the $(k+1)$-space $\langle \pi, \delta \rangle$, not containing $\delta$. Hence, $|S_1\setminus S_2|=\theta_{k+1}-\theta_{k-t}$.

Let $\sigma_0$ be a projective $(k-t)$-space in $\pi \setminus \delta$.
An element of $S_2$ has at least a point in common with $\sigma_0$. 
Hence, $|S_2|$ is the number of $k$-spaces through $\delta$, minus the number of $k$-spaces through $\delta$,
disjoint from $\sigma_0$. By investigating the quotient space of $\delta$, and due to Lemma \ref{lemmadisjunct}, 
we have that $|S_2|=\qbin{n-t}{k-t}-q^{(k-t+1)(k-t)}\qbin{n-k-1}{k-t} $. Hence

\begin{align}
|\mathcal{S}|
&=\theta_{k+1}-\theta_{k-t}+\qbin{n-t}{k-t}- q^{(k-t+1)(k-t)}\qbin{n-k-1}{k-t}. 
\end{align}

On the other hand, there also holds that $|\mathcal{S}|=\theta_{k+1}+\sum_{j=0}^{k-t-2} Q_j(n,k,t)$, with $Q_j(n,k,t)=|\{ \beta\in S_2| \beta \nsubseteq \langle \pi, \delta\rangle, \dim(\beta\cap \sigma_0)=j \}|$. Note that the first term is the number of $k$-spaces in $\langle \pi, \delta \rangle$.  Since there are $\qbin{k-t+1}{j+1}$ $j$-spaces in $\sigma_0$, and by using Lemma \ref{lemmadisjunct}, we find that 
\begin{align}\label{Sproj1somnotatie}
|\mathcal{S}|&=\theta_{k+1}+\sum_{j=0}^{k-t-2} \qbin{k-t+1}{j+1}q^{(k-t-j)(k-t-j-1)}\qbin{n-k-1}{k-t-j-1}.
\end{align}

	It is clear that all elements of $S_2$ pairwise meet in at least the $t$-space $\delta$. Every two elements of $S_1$ meet in a $(k-1)$-space, since they are contained in a $(k+1)$-space. Note that $k-1\geq t$. 
	Consider now a $k$-space $\pi_1$ in $S_1$ and a $k$-space $\pi_2$ in $S_2$. Note that $\pi_1 \subset  \langle \pi, \delta \rangle$, and $\pi_2$ meets $\langle \pi,\delta \rangle$ in at least a $(t+1)$-space. Again, due to the Grassmann dimension property, it follows that they meet in at least a $t$-space. 
	
	Now we prove that $\mathcal{S}$ cannot be extended to a larger set of $k$-spaces pairwise intersecting in at least a $t$-space. Suppose that $\alpha \notin \mathcal{S}$ is a $k$-space that meets every element of $\mathcal{S}$ in at least a $t$-space. If $\delta \subset \alpha$, then, since $\alpha \notin \mathcal{S}$, $\alpha$ meets $\langle \pi, \delta \rangle$ only in $\delta$. Hence, $\alpha$ meets $\pi$ in a $(t-1)$-space, which is a contradiction. This implies that we can suppose that $\delta \nsubseteq \alpha$. So, $\alpha$ meets $\delta$ in a $d$-space with $d\leq t-1$. Note that $\dim(\alpha\cap \langle \pi, \delta\rangle)\geq t+1$ since $\alpha$ meets all elements of $S_1$ in at least a $t$-space.
 For every point $P\in \sigma_0$, consider the set $\mathcal{S}_P$ of elements of $\mathcal{S}$ that meet $\langle \pi, \delta\rangle$ in $\langle \delta, P \rangle$. If $\dim (\alpha\cap \langle \delta, P \rangle)<t$, then $\alpha$ must meet all elements of $\mathcal{S}_P$ in a subspace outside of $\langle \pi, \delta\rangle$. This gives a contradiction since $n>2k-t$. Hence, $\dim (\alpha\cap \langle \delta, P \rangle)=t$ for all points $P\in \sigma_0$. This implies that $\dim(\alpha\cap \delta)=t-1$, and $\alpha$ must have a $t$-space in comon with all $(t+1)$-spaces $\langle \delta, P \rangle$ with $P\in \sigma_0$. Hence, $\alpha \in  \langle \pi, \delta \rangle$, and so $\alpha\in S_1$, a contradiction. 
\end{proof}

\begin{example}\label{example2p}
	Suppose $k>t$ and let $\Gamma$ be a $(t+2)$-space in $\PG(n,q)$, $n>2k-t$.  Let $\mathcal{S}$ be the set of all $k$-spaces in $\PG(n,q)$, meeting $\Gamma$ in at least a $(t+1)$-space.
\end{example}

\begin{lemma}
	The set $\mathcal{S}$, described in Example \ref{example2p}, is a maximal set of $k$-spaces in $\PG(n,q)$, pairwise intersecting in at least a $t$-space, of size \begin{align*}|\mathcal{S}| &=\qbin{n-t-2}{k-t-2} \left( 1+ \theta_{t+2}q^{k-t-1}\frac{q^{n-k}-1}{q^{k-t-1}-1}  \right).
	\end{align*}
\end{lemma}
\begin{proof}

The number of elements in $\mathcal{S}$ is the number of $k$-spaces through $\Gamma$,  together with the number of $k$-spaces, meeting $\Gamma$ in a $(t+1)$-space: \\
\begin{align*}|\mathcal{S}|&=\qbin{n-t-2}{k-t-2}+\theta_{t+2}\cdot\left(\qbin{n-t-1}{k-t-1}-\qbin{n-t-2}{k-t-2}\right)\\ &=\qbin{n-t-2}{k-t-2} \left( 1+ \theta_{t+2}q^{k-t-1}\frac{q^{n-k}-1}{q^{k-t-1}-1}  \right).
\end{align*}

	Consider two elements $\pi_1, \pi_2\in \mathcal{S}$. Then $\pi_1 \cap \Gamma$ and $\pi_2\cap \Gamma$ are two subspaces with dimension at least $t+1$ in a $(t+2)$-space, and so, they meet in at least a $t$-space. 
	
	Now we prove that $\mathcal{S}$ cannot be extended to a larger set of $k$-spaces pairwise intersecting in at least a $t$-space. Suppose that $\alpha \notin \mathcal{S}$ is a $k$-space that meets every element of $\mathcal{S}$ in at least a $t$-space. Since $\alpha \notin \mathcal{S}$, we know that $\dim(\alpha \cap \Gamma)\leq t$. Let $\gamma$ be a $(t+1)$-space in $\Gamma$ such that $\dim(\alpha \cap \Gamma\cap \gamma)\leq t-1$. 
	Then $\alpha$ must meet all elements of $\mathcal{S}$ through $\gamma$ in a subspace outside of $\Gamma$. 
	 For $n>2k-t$, this is not possible.
	 Hence, $\mathcal{S}$ cannot be extended. 
\end{proof}

\begin{remark}
	Note that for $k=t+1$, Example \ref{examplep} and Example \ref{example2p} coincide. In that case, $\mathcal{S}$ is the set of all $(t+1)$-spaces in a fixed $(t+2)$-space in $\PG(n,q)$.
\end{remark}

\section{Two  examples in $\AG(n,q)$}\label{sectionexampleaffien}
We also give two examples of maximal sets of $k$-spaces in AG$(n,q)$, pairwise meeting in at least a $t$-space. In Section \ref{sectionclassi}, we prove that the largest non-trivial sets of $k$-spaces, pairwise meeting in at least a $t$-space, in AG$(n,q)$ are given by Example \ref{example} or Example \ref{example2}. If $k>2t+1$, Example \ref{example} is the largest set, whereas if $k\leq 2t+1$, Example \ref{example2} is the largest one. 

From now on, for an affine subspace $\alpha$ we denote the projective extension of $\alpha$ by $\tilde{\alpha}$, and let $H_\infty=\PG(n,q)\setminus \AG(n,q)$ be the hyperplane at infinity. Similarly, if $\mathcal{S}$ is a set of affine spaces, then we denote the corresponding set of projective spaces by $\tilde{\mathcal{S}}$.

Let $\mathcal{S}$ be a set of $k$-spaces in $\AG(n,q)$ with the property that for every two elements $\pi_1, \pi_2 \in \mathcal{S}$, $\tilde{\pi}_1\cap H_\infty \neq \tilde{\pi}_2\cap H_\infty$. Then we find, for $n\leq 2k-t$, that $\mathcal{S}$  is a set of $k$-spaces pairwise intersecting in at least an affine $t$-space. Hence, we suppose from now on, that $n>2k-t$.

\begin{example}\label{example}
	Let $\delta$ be a $t$-space, $t<k$, in $\AG(n,q)$, and let $\pi$ be a $k$-space in AG$(n,q)$ with $\pi\cap \delta$ an affine $(t-1)$-space. 
	Let $S_1$ be a \textcolor{black}{maximal} set of 
	 affine $k$-spaces in $\langle \pi, \delta \rangle$, including $\pi$, such that for any two elements $\pi_1, \pi_2$ of $S_1$, $\tilde{\pi}_1\cap H_\infty \neq \tilde{\pi}_2\cap H_\infty$, and such that for every $\pi_1 \in S_1$: $ \tilde{\delta}\cap H_\infty \nsubseteq\tilde{\pi}_1$.
	Let $S_2$ be the set of all $k$-spaces through $\delta$ and meeting $\langle \pi, \delta \rangle$ in at least a $(t+1)$-space.
	Let $\mathcal{S}$ be the union of the sets $S_1$ and $S_2$. 
\end{example}

\begin{lemma}
	The set $\mathcal{S}$, described in Example \ref{example}, is a {maximal} set of $k$-spaces in $\AG(n,q)$, $n>2k-t$, pairwise intersecting in at least a $t$-space, of size 	\begin{align}
	|\mathcal{S}|
	&=\theta_{k}-\theta_{k-t}+\qbin{n-t}{k-t}- q^{(k-t+1)(k-t)}\qbin{n-k-1}{k-t}\\
	&=\theta_{k}+\sum_{j=0}^{k-t-2} \qbin{k-t+1}{j+1}q^{(k-t-j)(k-t-j-1)}\qbin{n-k-1}{k-t-j-1}. 
	\end{align}
\end{lemma}
\begin{proof}
We start with determining the size of $\mathcal{S}$. Note first that the number of elements of $S_1$ is equal to the number of $(k-1)$-spaces in $H_\infty\cap \langle \pi, \delta \rangle$, not containing $\delta\cap H_\infty$. Hence, $|S_1|=\theta_k-\theta_{k-t}$.

	Let $\tilde{\sigma_0}$ be a projective $(k-t)$-space in $\tilde{\pi} \setminus \tilde{\delta}$.
	An extended element of $S_2$ to $\PG(n,q)$, has at least a point in common with $\tilde{\sigma_0}$. 
	Hence, $|S_2|$ is the number of projective $k$-spaces through $\delta$, minus the number of $k$-spaces through $\delta$, disjoint to $\tilde{\sigma_0}$. By investigating the quotient space of $\tilde{\delta}\cap \tilde{\pi}$, and due to Lemma \ref{lemmadisjunct}, we have that $|S_2|=\qbin{n-t}{k-t}-q^{(k-t+1)(k-t)}\qbin{n-k-1}{k-t} $. Hence,
	
	\begin{align}
	|\mathcal{S}|
&=\theta_{k}-\theta_{k-t}+\qbin{n-t}{k-t}- q^{(k-t+1)(k-t)}\qbin{n-k-1}{k-t} 
	\end{align}
	
	On the other hand, we find that $|\mathcal{S}|=\theta_{k}+\sum_{j=0}^{k-t-2} Q_j(n,k,t)$, with $Q_j(n,k,t)=\{ \beta\in S_2| \beta \nsubseteq \langle \pi, \delta\rangle, \delta\subset \beta, \dim(\beta\cap (\pi\setminus \delta))=j \}$. Note that the first term is the number of $k$-spaces in $\langle \pi, \delta \rangle$ such that any two of these elements meet in an affine $(k-1)$-space.  Since there are $\qbin{k-t+1}{j+1}$ $j$-spaces in $\pi \setminus \delta$, and by using Lemma \ref{lemmadisjunct}, we find that 
	\begin{align}\label{Saffien1somformule}
	|\mathcal{S}|&=\theta_{k}+\sum_{j=0}^{k-t-2} \qbin{k-t+1}{j+1}q^{(k-t-j)(k-t-j-1)}\qbin{n-k-1}{k-t-j-1}.
	\end{align}


	It is clear that all elements of $S_2$ pairwise meet in at least a $t$-space ($\delta$). Consider now two elements $\pi_1, \pi_2\in S_1$. It follows, from the Grassmann dimension property,  that $\tilde{\pi}_1\cap \tilde{\pi}_2$ is a $(k-1)$-space in the $(k+1)$-space $\langle \tilde{\pi}, \tilde{\delta}\rangle$. This $(k-1)$-space is not contained in $H_\infty$ by the definition of $S_1$.
	Consider now a $k$-space $\pi_1$ in $S_1$ and a $k$-space $\pi_3$ in $S_2$. Note that $\tilde{\pi}_1 \subset \langle \tilde{\pi}, \tilde{\delta}\rangle$, and $\tilde{\pi}_3$ meets $\langle \tilde{\pi}, \tilde{\delta}\rangle$ in at least a $(t+1)$-space. Again, due to the Grassmann dimension property, it follows that $\tilde{\pi}_1\cap \tilde{\pi}_3$ is at least a $t$-space in $\langle \tilde{\pi}, \tilde{\delta}\rangle$. This $t$-space is not contained in $H_\infty$ as $\tilde{\pi}_1$ does not contain $\tilde{\delta}\cap H_\infty$.
	
	Now we prove that $\mathcal{S}$ cannot be extended to a larger set of $k$-spaces pairwise intersecting in at least an affine $t$-space. Suppose that $\alpha \notin \mathcal{S}$ is an affine $k$-space that meets every element of $\mathcal{S}$ in at least an affine $t$-space. If $\alpha$ contains $\delta$, then, since $\alpha \notin \mathcal{S}$, we know that $\alpha \cap \langle \pi, \delta \rangle=\delta$. Hence, $\alpha$ meets $\pi$ only in a $(t-1)$-space, a contradiction. So we can suppose that  $\delta \nsubseteq \alpha$.
	 Hence, $\dim(\alpha \cap \delta)\leq t-1$, and note that $\dim(\alpha\cap \langle \pi, \delta\rangle)\geq t+1$ as $\alpha$ meets all elements of $S_1$ in at least a $t$-space.  Let $\sigma_0$ be a $(k-t)$-space in $\pi\setminus \delta$.  For every point $P\in \sigma_0$, consider the set $\mathcal{S}_P$ of elements of $\mathcal{S}$ that meet $\langle \pi, \delta\rangle$ in $\langle \delta, P \rangle$. If $\dim (\alpha\cap \langle \delta, P \rangle)<t$, then $\alpha$ must meet all elements of $\mathcal{S}_P$ in a subspace outside of $\langle \pi, \delta\rangle$. This gives a contradiction since $n>2k-t$. Hence, $\dim (\alpha\cap \langle \delta, P \rangle)=t$ for all points $P\in \sigma_0$. This implies that $\dim(\alpha\cap \delta)=t-1$, and $\alpha$ must have a $t$-space in comon with all $(t+1)$-spaces $\langle \delta, P \rangle$ with $P\in \sigma_0$. Hence, $\alpha \in  \langle \pi, \delta \rangle$, and so $\alpha\in S_1$, a contradiction. 
	
\end{proof}

\begin{example}\label{example2}
	Suppose $k>t$. Let $\Gamma$ be an affine $(t+2)$-space in $AG(n,q)$, and let $\mathcal{R}$ be a maximal set of $\theta_{t+1}$ affine $(t+1)$-spaces in $\Gamma$ such that for every two distinct elements $\sigma_1, \sigma_2\in \mathcal{R}$, $\tilde{\sigma}_1\cap H_\infty \neq \tilde{\sigma}_2\cap H_\infty$. Note that every two different elements of $R$ meet in an affine $t$-space. Let $\mathcal{S}$ be the set of all $k$-spaces in $AG(n,q)$, containing $\Gamma$ or  meeting $\Gamma$ in an element of $\mathcal{R}$.
\end{example}

\begin{lemma}
	The set $\mathcal{S}$, described in Example \ref{example2}, is a {maximal} set of $k$-spaces in $\AG(n,q)$, $n>2k-t$, pairwise intersecting in at least a $t$-space, of size 
	\begin{align*}|\mathcal{S}|
	&=\qbin{n-t-2}{k-t-2} \left( 1+ \theta_{t+1}q^{k-t-1}\frac{q^{n-k}-1}{q^{k-t-1}-1}  \right).
	\end{align*}
\end{lemma}
\begin{proof}
	Since $\mathcal{R}$ is a maximal set, we have that $|\mathcal{R}|$ is the number of all $t$-spaces in $\tilde{\Gamma} \cap H_\infty$. Hence,  $|\mathcal{R}|=\theta_{t+1}$. 
	The number of elements in $\mathcal{S}$ is the number of $k$-spaces through $\Gamma$,  together with the number of $k$-spaces, meeting $\Gamma$ in an element of $\mathcal{R}$: \\
 \begin{align*}|\mathcal{S}|&=\qbin{n-t-2}{k-t-2}+\theta_{t+1}\cdot\left(\qbin{n-t-1}{k-t-1}-\qbin{n-t-2}{k-t-2}\right)\\ 
 &=\qbin{n-t-2}{k-t-2} \left( 1+ \theta_{t+1}q^{k-t-1}\frac{q^{n-k}-1}{q^{k-t-1}-1}  \right).
\end{align*}

	Consider two elements $\pi_1, \pi_2\in \mathcal{S}$. If $\pi_1$ or $\pi_2$ contains $\Gamma$, then $\pi_1$ and $\pi_2$ intersect in at least a $(t+1)$-dimensional space. Hence, we suppose that $\pi_1 \cap \Gamma$ and $\pi_2\cap \Gamma$ are two $(t+1)$-spaces of $\mathcal{R}$ in a $(t+2)$-space. Since every two elements of $\mathcal{R}$ meet  in an affine space with dimension at least $t$, we have that $\pi_1$ and $\pi_2$ meet in at least an affine $t$-space. 
	
	Now we prove that $\mathcal{S}$ cannot be extended to a larger set of $k$-spaces pairwise intersecting in at least a $t$-space. Suppose that $\alpha \notin \mathcal{S}$ is an affine $k$-space that meets every element of $\mathcal{S}$ in at least a $t$-space.
	Consider an element $\sigma\in \mathcal{R}$. Since $\alpha$ must meet all affine $k$-spaces through $\sigma$, we find that $\alpha$ contains a $t$-space of $\sigma$, as $n>2k-t$. As $\sigma$ is an arbitrary element of $\mathcal{R}$, we see that $\alpha$ must meet every element of $\mathcal{R}$ in at least an affine $t$-space. This implies that $\alpha$ meets $\Gamma$ in a $(t+1)$-space $\alpha_\Gamma$. Now, $\alpha_\Gamma$ must meet every element of $\mathcal{R}$ in an affine $t$-space. Due to the maximality of $\mathcal{R}$, we have that $\alpha_\Gamma \in \mathcal{R}$, and so that $\alpha \in \mathcal{S}$, a contradiction.
	
\end{proof}


\section{Classification results}\label{sectionclassi}
We start with a classification result on maximal sets of $k$-spaces pairwise intersecting in a $(k-1)$-space. In the projective case, the following classification result was proven by Brouwer, Cohen and Neumaier, based on the link with distance regular graphs.

\begin{theorem}{\textbf{\cite[Section 9.3]{brouwer}}}\label{basisEKRthm}
	Let $\mathcal{S}$ be a set of projective $k$-spaces, pairwise intersecting in a $(k-1)$-space in $\PG(n,q)$, $n\geq k+2$, then all the $k$-spaces of $\mathcal{S}$ go through a fixed $(k-1)$-space or they are contained in a fixed $(k + 1)$-space.
\end{theorem}

We used the above classification to deduce the classification of maximal sets of $k$-spaces pairwise intersecting in a $(k-1)$-space in $\AG(n,q)$. 

\begin{theorem}
	Let $\mathcal{S}$ be a set of $k$-spaces in $\AG(n,q)$, $n>k$, pairwise intersecting in a $(k-1)$-space such that $\mathcal{S}$ is not a $(k-1)$-pencil, 
	then $|\mathcal{S}|\leq \theta_k$, and equality occurs if and only if 
	$\mathcal{S}$ is Example \ref{example2}. Hence, 
	all elements of $\mathcal{S}$ are contained in a $(k+1)$-space.
\end{theorem}
\begin{proof}
	As before, Let $\tilde{\mathcal{S}}$ be the set of  projective extensions of the elements in $\mathcal{S}$.  So, $\tilde{\mathcal{S}}$ is a set of projective $k$-spaces pairwise intersecting in a $(k-1)$-space, and such that there is no $(k-1)$-space contained in all these elements. Hence, $\tilde{\mathcal{S}}$ is contained in a $(k+1)$-space $\Pi$, by Theorem \ref{basisEKRthm}. 
	Now, every two elements of $\mathcal{S}$ must meet in $\AG(n,q)$. So, for every two elements $\pi_1,\pi_2 \in \mathcal{S}$, $\tilde{\pi}_1\cap \tilde{\pi}_2 \nsubseteq H_\infty$. This implies that every $k$-space in $\Pi\cap H_\infty$ is contained in precisely one element of $\tilde{\mathcal{S}}$. This is Example \ref{example2}, for $k=t+1$, which proves the theorem.
\end{proof}
\begin{remark}
	Note that for $t=k-1$, Example \ref{example} and Example \ref{example2} are similar. In both examples, all elements are contained in a fixed $(t+2)$-space $\Gamma$, and every projective $t$-space in $\tilde{\Gamma} \cap H_\infty$ is contained in a unique (extended) $k$-space of the set.  The two examples are not equal, since in Example \ref{example}, there is freedom to choose the affine $k$-space $\pi$, such that $\tilde{\pi}$ goes through a fixed $t$-space in $H_\infty$. 
\end{remark}
For $t=k-1$, the number of elements of Example \ref{example} is $\theta_{k}$, while, the number of affine subspaces in $\AG(n,q)$ through a fixed affine $(k-1)$-space is $\theta_{n-k}$. Hence, for $n<2k$, Example \ref{example} is the largest example of a set of affine $k$-spaces, pairwise intersecting in at least a $(k-1)$-space.\\

From now on, we suppose that $k>t+1$. In Section \ref{sectionproj} and Section \ref{sectionaffien}, we classify the largest non-trivial $t$-intersecting sets of $k$-spaces in $\PG(n,q)$ and $\AG(n,q)$, respectively. Several ideas in the following subsection are based on the notes of David Ellis \cite{ellis}.

\subsection{Classification result in $\PG(n,q)$}\label{sectionproj}

Let $\mathcal{S}_p$
 be a maximal set of $k$-spaces in $\PG(n,q)$, $n>2k-t$, $k>t+1$, pairwise meeting in at least a $t$-space. Let 
\begin{align*}
	 \psi(\mathcal{S}_p)=\min\{\dim(T)| T\subset \PG(n,q), \dim(T\cap \alpha)\geq t, \ \forall \alpha \in \mathcal{S}_p \}. \\
\end{align*}
Note that $\psi(S_p)$ is well-defined. Every element $\beta\in \mathcal{S}_p$ is an example of a subspace such that $\dim(\beta\cap \alpha)\geq t, \forall \alpha \in \mathcal{S}_p$.
Let $\mathcal{T}$ be the collection of all $\psi(\mathcal{S}_p)$-dimensional spaces in $\PG(n,q)$, that meet every element of $\mathcal{S}_p$ in at least a $t$-space. 
\begin{lemma}\label{remarks}
	\begin{enumerate}
		\item We have that $t\leq\psi(\mathcal{S}_p)\leq k$, and if $\psi(\mathcal{S}_p)=t$, then $\mathcal{S}_p$ is a $t$-pencil. 
		\item If $T\in  \mathcal{T}$, then all $k$-spaces through $T$ are contained in $\mathcal{S}_p$.
		\item The elements of  $\mathcal{T}$ are $t$-intersecting in $\PG(n,q)$.
	\end{enumerate}
\end{lemma}
\begin{proof}
	\begin{enumerate}
		\item Let $\pi_1\in \mathcal{S}_p$. Since every element of $\mathcal{S}_p$ meets $\pi_1$ in at least a $t$-space, we have that $\psi(\mathcal{S}_p)\leq k$. Let $T\in \mathcal{T}$. Since all elements of $\mathcal{S}_p$ meet $T$ in at least a $t$-space, we have that $\psi(\mathcal{S}_p)\geq t$. If $\psi(\mathcal{S}_p)=t$, then all elements of $\mathcal{S}_p$ contain the $t$-space $T$, and hence, $\mathcal{S}_p$ is a $t$-pencil.
		\item This property follows from the maximality of $\mathcal{S}_p$.
		\item Suppose that there are two elements $T_1, T_2 \in \mathcal{T}$, with $\dim(T_1\cap T_2)=l<t$. Let $\pi_i$ be a $k$-space through $T_i$, $i=1,2$, such that $\dim(\pi_1 \cap \pi_2)<t$. Note that we can find these $k$-spaces since $n>2k-t$. From the second item, we have that $\pi_1, \pi_2\in \mathcal{S}_p$, a contradiction since they have no $t$-space in common.
	\end{enumerate}
\end{proof}
\begin{lemma}\label{lemmacounting1}
	Let $\psi(\mathcal{S}_p)=t+x,\  x\geq2$, $k>t+1$ and $n> 2k-t$. Then the number of elements of $\mathcal{S}_p$ through a projective $(t+x-j)$-space, with $j\in \{0,1,2,\dots, x\}$, is at most
	$(\theta_{k-t})^j \qbin{n-t-x}{k-t-x}.$ 
\end{lemma}
\begin{proof} 
	Let $\psi(\mathcal{S}_p)=t+x, \ x\geq 2$. 
	We prove, by induction on $j\in \{0,1,2,\dots, x\}$, that the number of $k$-spaces of ${\mathcal{S}_p}$  through a  $(t+x-j)$-space is at most $ \qbin{n-t-x}{k-t-x}(\theta_{k-t})^j$. Note that, by Lemma \ref{remarks}$(2)$, the statement is true for $j=0$. 
	Let  $j\in \{1,2, 3,\dots, x\}$ and suppose now that the number of $k$-spaces of ${\mathcal{S}_p}$ through a projective $(t+x-j_0)$-space, is at most $(\theta_{k-t})^{j_0} \qbin{n-t-x}{k-t-x}$, for all $j_0<j$. Then we prove that this also holds for $j$.
	Consider a projective $(t+x-j)$-space $\gamma_j$. 
	Since $\psi(\mathcal{S}_p)=t+x$, we know that there exists a $k$-space $\pi_j$ of ${\mathcal{S}_p}$, meeting $\gamma_j$ in at most a $(t-1)$-space. 
	Suppose first that $\dim(\pi_j\cap \gamma_j)=t-1$ and let $\pi_{j\gamma}$ be a projective $(k-t)$-space in $\pi_j \setminus \gamma_j$. Then every element of ${\mathcal{S}_p}$ through $\gamma_j$ contains at least a point of $\pi_{j\gamma}$. Since there are $\theta_{k-t}$ points in $\pi_{j\gamma}$, and since the number of projective $k$-spaces through a $(t+x-j+1)$-space is at most $(\theta_{k-t})^{j-1}\qbin{n-t-x}{k-t-x}$, we find that the number of elements of ${\mathcal{S}_p}$ through $\gamma_j$ is at most $ (\theta_{k-t})^{j} \qbin{n-t-x}{k-t-x}.$
	
Suppose now that every element $\pi$ of ${\mathcal{S}_p}$ meets $\gamma_j$ in a $t$-space or in at most a $(t-2)$-space. 
Let $\max\{\dim(\gamma_j\cap \pi)|\pi \in {\mathcal{S}_p}, \dim(\gamma_j\cap \pi)< t \}=t-l$, then $l\geq 2$,  and suppose that $\pi_j\in \mathcal{S}_p$ is an element such that $\dim(\pi_j\cap \gamma_j)=t-l$. Let $\pi_{j\gamma}$ be a projective $(k-t+l-1)$-space in $\pi_j \setminus \gamma_j$.  Then every element of $\mathcal{S}_p$
through $\gamma_j$ contains at least an $(l-1)$-space of $\pi_{j\gamma}$. Since there are $\qbin{k-t+l}{l}$ $(l-1)$-spaces in $\pi_{j\gamma}$, and since the number of projective $k$-spaces through a $(t+x-j+l)$-space is at most $(\theta_{k-t})^{j-l}\qbin{n-t-x}{k-t-x}$, we find that the number of elements of ${\mathcal{S}_p}$ through $\gamma_j$ is at most 
 $ \qbin{k-t+l}{l} (\theta_{k-t})^{j-l}\qbin{n-t-x}{k-t-x}$. Note that 
\begin{align*}
	\qbin{k-t+l}{l} (\theta_{k-t})^{j-l}= \frac{(q^{k-t+l}-1)\dots (q^{k-t+1}-1)}{(q^l-1)\dots (q-1)} (\theta_{k-t})^{j-l} \leq  \left(\frac{(q^{k-t+1}-1)}{(q-1)}\right)^l (\theta_{k-t})^{j-l}=(\theta_{k-t})^{j}
\end{align*}
		Hence, also in this case, 
		we find that the number of elements of ${\mathcal{S}_p}$ through $\gamma_j$ is at most $ (\theta_{k-t})^{j} \qbin{n-t-x}{k-t-x}.$
\end{proof}

\begin{lemma}\label{lemmacountingproj}
	Let $\mathcal{S}_p$ be a set of $k$-spaces, pairwise intersecting in at least a $t$-space in $\PG(n,q)$. If $\psi(\mathcal{S}_p)=t+x,\  x\geq2$, $k>t+1$ and $n> 2k-t$, then 
	$|\mathcal{S}_p|\leq (\theta_{k-t})^x\qbin{n-t-x}{k-t-x}\qbin{t+x+1}{t+1}  .$ 
\end{lemma}
\begin{proof}
	Suppose that $\psi(\mathcal{S}_p)=t+x, \ x\geq 2$. By Lemma \ref{lemmacounting1}, we know, for  $j\in \{0,1,2,\dots, x\}$, that the number of $k$-spaces of ${\mathcal{S}_p}$  through a  $(t+x-j)$-space is at most $ \qbin{n-t-x}{k-t-x}(\theta_{k-t})^j$. 
	
	Consider now an element $T\in \mathcal{T}$. Then every element of $\mathcal{S}_p$ meets $T$ in at least a $t$-space. Since there are $ \qbin{t+x+1}{t+1}$ projective $t$-spaces in $T$ and since every $t$-space is contained in at most $(\theta_{k-t})^{x} \qbin{n-t-x}{k-t-x}$ elements of ${\mathcal{S}_p}$, we find that  $\mathcal{S}_p$ has at most $(\theta_{k-t})^{x} \qbin{n-t-x}{k-t-x}\qbin{t+x+1}{t+1}$ elements.
	\end{proof}

\begin{lemma}\label{lemma1ellis}
	Let $\mathcal{S}_p$ be a maximal set of $k$-spaces, pairwise intersecting in at least a $t$-space in $\PG(n,q)$, {$n> 2k-t, k>t$}. If $\psi(\mathcal{S}_p)=t+1$ and $|\mathcal{T}|\leq 2$, then \[|\mathcal{S}_p|\leq 2\qbin{n-t-1}{k-t-1}+ (\theta_{t+1}\theta_{k-t}-\theta_{t+1}-1) \theta_{k-t} \qbin{n-t-2}{k-t-2}.\]
\end{lemma}
\begin{proof}
	Let $T$ be a $(t+1)$-space  of $\mathcal{T}$. Since $\mathcal{S}_p$ is a maximal set, we know that all $\qbin{n-t-1}{k-t-1}$ $k$-spaces through $T$ are contained in $\mathcal{S}_p$. Now we count the size of the set $\mathcal{S}_{p0}$ of $k$-spaces of $\mathcal{S}_p$ not through $T$. For every $\pi\in \mathcal{S}_{p0}$, $\dim(\pi\cap T)=t$. Let $E$ be a $t$-space in $T$, then there exists an element $\alpha\in \mathcal{S}_{p0}$ not through $E$, and so $\dim(\alpha \cap E)=t-1$. Hence, every element $\pi$ of $\mathcal{S}_{p0}$, through $E$ must contain a $(t+1)$-space $\tau$, different from  $T$, such that $E\subset \tau$ and $\tau\cap (\alpha \setminus E) \neq \emptyset$. Note that there are $\theta_{k-t}-1$ possibilities for $\tau$. Fix such a $(t+1)$-space $\tau$. 
	\begin{itemize}
		\item If $\mathcal{T}=\{T\}$, we know that $\tau \notin \mathcal{T}$, and hence there exists an element $\sigma$ of $\mathcal{S}_p$, meeting $\tau$ in at most a $(t-1)$-space. Hence, every element of $\mathcal{S}_{p0}$ through $\tau$ meets $\sigma \setminus \tau$, and so the number of elements of $\mathcal{S}_{p0}$ through $\tau$ is at most $\theta_{k-t} \qbin{n-t-2}{k-t-2}$. 
		Since there are $\theta_{t+1}$ possibilities for $E$, and at most $\theta_{k-t} -1$ for $\tau$, we have that \[|\mathcal{S}_p|\leq \qbin{n-t-1}{k-t-1}+\theta_{t+1} (\theta_{k-t}-1)\theta_{k-t} \qbin{n-t-2}{k-t-2}.\]
		\item Suppose $|\mathcal{T}|=2$, and let $\mathcal{T}=\{T, \Psi\}$. If $\tau=\Psi$, then $\mathcal{S}_p$ contains all $\qbin{n-t-1}{k-t-1}$ $k$-spaces through $\tau$. If $\tau\neq \Psi$, then we can follow the argument in the previous item, and we find that the number of elements of $\mathcal{S}_{p0}$ through $\tau$ is at most $\theta_{k-t} \qbin{n-t-2}{k-t-2}$. 
		Note that there are $\theta_{t+1}-1$ possibilities for $E\neq T\cap \Psi$. If $E\neq T\cap \Psi$, there are at most $\theta_{k-t} -1$ possibilities for $\tau\neq \Psi, T$, through $E$.  Next to this, if $E= T\cap \Psi$, there are at most $\theta_{k-t} -2$ possibilities for $\tau\neq \Psi$ through $E= T\cap \Psi$. Hence, we have that \begin{align*}
			|\mathcal{S}_p|\leq& \qbin{n-t-1}{k-t-1}+ \sum_{E\subset T}\sum_{\tau \supseteq E} |\{\pi \in \mathcal{S}_{p0}|\tau \subset \pi\}| \\
			\leq & \qbin{n-t-1}{k-t-1}+ \sum_{E\neq T\cap \Psi }\sum_{\tau \supseteq E} \theta_{k-t} \qbin{n-t-2}{k-t-2}+\sum_{\tau \supset T\cap \Psi} |\{\pi \in \mathcal{S}_{p0}|\tau \subset \pi\}|\\
			\leq & \qbin{n-t-1}{k-t-1}+ (\theta_{t+1}-1)(\theta_{k-t}-1) \theta_{k-t} \qbin{n-t-2}{k-t-2}+\sum_{\tau \neq  \Psi} \theta_{k-t}\qbin{n-t-2}{k-t-2}+\qbin{n-t-1}{k-t-1}\\
			\leq &  2\qbin{n-t-1}{k-t-1}+ (\theta_{t+1}-1)(\theta_{k-t}-1) \theta_{k-t} \qbin{n-t-2}{k-t-2}+(\theta_{k-t}-2) \theta_{k-t}\qbin{n-t-2}{k-t-2}\\
		= &  2\qbin{n-t-1}{k-t-1}+ (\theta_{t+1}\theta_{k-t}-\theta_{t+1}-1) \theta_{k-t} \qbin{n-t-2}{k-t-2}.\\
		\end{align*}
	\end{itemize}
The lemma follows since \[2\qbin{n-t-1}{k-t-1}+ (\theta_{t+1}\theta_{k-t}-\theta_{t+1}-1) \theta_{k-t} \qbin{n-t-2}{k-t-2}\geq \qbin{n-t-1}{k-t-1}+\theta_{t+1} (\theta_{k-t}-1)\theta_{k-t} \qbin{n-t-2}{k-t-2}\] for $n>2k-t, k>t, q\geq 2$ (see Lemma \ref{Lemma 47}).
\end{proof}

From now on, we define $f_p(q,n,k,t)$ as the maximum of the number of elements in the sets described in Example \ref{examplep} and Example \ref{example2p}.
\begin{align*}
f_p(q,n,k,t)=&\max\{\theta_{k+1}-\theta_{k-t}+\qbin{n-t}{k-t}- q^{(k-t+1)(k-t)}\qbin{n-k-1}{k-t} ,\\ &\qquad \quad \theta_{t+2}\cdot\left(\qbin{n-t-1}{k-t-1}-\qbin{n-t-2}{k-t-2}\right)+\qbin{n-t-2}{k-t-2} \}.\\
\end{align*}

From Lemma \ref{appendixprojverschil1}, \ref{appendixprojverschil2} and \ref{appendixprojverschil3}, we find, for $n>2k-t, k>t+1, q\geq 3$, that 

\begin{align*}
	f_p(q,n,k,t)=&\begin{cases}
	\theta_{k+1}-\theta_{k-t}+\qbin{n-t}{k-t}- q^{(k-t+1)(k-t)}\qbin{n-k-1}{k-t} &\text{if }{ k>2t+2} \\
	\theta_{t+2}\cdot\left(\qbin{n-t-1}{k-t-1}-\qbin{n-t-2}{k-t-2}\right)+\qbin{n-t-2}{k-t-2} &\text{if } k\leq2t+2.
	\end{cases}	\\
\end{align*}

\begin{theorem}
	Let $\mathcal{S}_p$ be a maximal set of $k$-spaces, pairwise intersecting in at least a $t$-space in $\PG(n,q)$, {$ k>t+1$, $t>0$,} with $q\geq 4$, and $n> 2k+t+2,$ (or $q=3$ and $n> 2k+t+3$). If $\mathcal{S}_p$ is not a $t$-pencil, then 
	\begin{align*}
	|\mathcal{S}_p|\leq f_p(q,n,k,t).
	\end{align*}
	Equality occurs if and only if $\mathcal{S}_p$ is Example \ref{examplep} for $k>2t+2$ or Example \ref{example2p} for {$k\leq 2t+2$}.
\end{theorem}
\begin{proof}
	Let $\mathcal{S}_p$ be a maximal set of $k$-spaces, pairwise intersecting in at least a $t$-space, in $\PG(n,q)$, with $\mathcal{S}_p$ not a $t$-pencil, and suppose that $|\mathcal{S}_p|\geq f_p(q,n,k,t)$. From Lemma \ref{lemmacountingproj} and Lemma \ref{lemmaappendixlelijk}, it follows that $\psi(\mathcal{S}_p)<t+2$ for  {$ k>t+1, t>0$,} and $q\geq 4, n>2k+t+2$, (or $q=3$ and $n> 2k+t+3$). Since $\mathcal{S}_p$ is not a $t$-pencil, $\psi(\mathcal{S}_p)>t$, and so $\psi(\mathcal{S}_p)=t+1$.
From Lemma \ref{lemma1ellis}, it follows that if $|{\mathcal{T}}|\leq 2$,
 then $|{\mathcal{S}_p}|\leq 2\qbin{n-t-1}{k-t-1}+ (\theta_{t+1}\theta_{k-t}-\theta_{t+1}-1) \theta_{k-t} \qbin{n-t-2}{k-t-2}$,
 a contradiction for $n>2k+t+2, q\geq 3, k>t+1, t>0$, due to Lemma \ref{appendixgeennaammeer}. Hence, $|{\mathcal{T}}|>2$. 
	From Lemma \ref{remarks}$(3)$, it follows that $\mathcal{T}$ is a $t$-intersecting set of $(t+1)$-spaces. Hence, 
	$\mathcal{T}$ is contained in a $t$-pencil or all elements of $\mathcal{T}$ are contained in a $(t+2)$-space (see Theorem \ref{basisEKRthm}).
	
	We first suppose that there is no $t$-space contained in all elements of  $\mathcal{T}$. Hence, we know that all elements of $\mathcal{T}$ are contained in a $(t+2)$-space $\Gamma$. 
This implies that every element $\pi_1$ of $\mathcal{S}_p$ must meet $\Gamma$ in at least a $(t+1)$-space.
	Since $\mathcal{S}_p$ is maximal, we know that $\mathcal{S}_p$ contains all $k$-spaces meeting $\Gamma$ in at least a $(t+1)$-space, which is Example \ref{example2p}. Hence,  $|\mathcal{S}_p|= \theta_{t+2}\cdot\left(\qbin{n-t-1}{k-t-1}-\qbin{n-t-2}{k-t-2}\right)+\qbin{n-t-2}{k-t-2}$, if there is no $t$-space contained in all elements of $\mathcal{T}$.  This number is larger than $\theta_{k+1}-\theta_{k-t}+\qbin{n-t}{k-t}- q^{(k-t-1)(k-t)}\qbin{n-k-1}{k-t}$, if and only if $k\leq 2t+2$. So, for $k>2t+2$, we find a contradiction. 
	
	Hence, for $k>2t+2$, we know that the elements of $\mathcal{T}$ are contained in a $t$-pencil with vertex the $t$-space $\delta$. 
	Let $Z$ be the span of all elements of $\mathcal{T}$ and let $\dim(Z)=t+x$, $x\geq 2$. Since $\mathcal{S}_p$ is not a $t$-pencil, we know that there are $k$-spaces in $\mathcal{S}_p$ that do not contain $\delta$. Note that these elements of $\mathcal{S}_p$, not through $\delta$, meet $\delta$ in a $(t-1)$-space, and meet $Z$ in a $(t+x-1)$-space. 
	The dimension of the span $Z$ of all the $(t+1)$-spaces in $\mathcal{T}$ is at most $k+1$:  if $\dim(Z)>k+1$ then every $k$-space of $\mathcal{S}_p$, not through $\delta$ would meet $Z$ in a subspace with dimension $\dim(Z)-1>k$, a contradiction.
	
	Let $\pi\in \mathcal{S}_p$ be an element that does not contain $\delta$. Note that every element of $\mathcal{S}_p$ through $\delta$ has at least a $(t+1)$-space in common with $\langle \pi, \delta \rangle$.
	Now we claim that all elements of $\mathcal{S}_p$, not through $\delta$, are contained in $\langle \pi, \delta \rangle$.
	Suppose that this is not the case, then there exists an element $\pi_2\in \mathcal{S}_p$ with $\delta \nsubseteq \pi_2$ and $\pi_2 \nsubseteq\langle \pi,\delta \rangle$. Then every element of $\mathcal{S}_p$ through $\delta$ meets both $\pi\setminus \delta$ and $\pi_2\setminus \delta$. Hence, the number of elements of $\mathcal{S}_p$, through $\delta$, is at most $\theta_{k-t}^{2}\qbin{n-t-2}{k-t-2}+\theta_{k-t-1}\qbin{n-t-1}{k-t-1}$. Here, the first term is an upper bound on the number of elements meeting both $\pi \setminus \pi_2$ and $\pi_2 \setminus \pi$. The second term is an upper bound on the number of elements meeting $(\pi\cap \pi_2) \setminus \delta$. Since every element of $\mathcal{S}_p$ not through $\delta$ meets $Z$ in a $(t+x-1)$-space, we find that $|\mathcal{S}_p|\leq \theta_{t+x}\qbin{n-t-x+1}{k-t-x+1}+\theta_{k-t}^{2}\qbin{n-t-2}{k-t-2}+\theta_{k-t-1}\qbin{n-t-1}{k-t-1}$. For $2\leq x\leq k-t+1, k>2t+2, n>2k+t+2, t>0$ and $q\geq 3$; this gives a contradiction by Lemma \ref{appendixlaatste}, since $|S|\geq f_p(q,n,k,t)$. Hence, we find that $\mathcal{S}_p$ is contained in Example \ref{examplep}. The theorem follows from the maximality of $\mathcal{S}_p$.
	\end{proof}

\subsection{Classification result in $\AG(n,q)$}\label{sectionaffien}

In this subsection, we investigate the largest non-trivial sets of $k$-spaces in $\AG(n,q)$ pairwise intersecting in at least a $t$-space. Many results and proofs in this affine setting are similar to the results and proofs in the projective setting, but because of some structural differences, we decided to discuss the Hilton-Milner problem, in the projective and affine context,  in different subsections.

We again suppose that $k>t+1$. Let $\mathcal{S}_a$
be a maximal set of $k$-spaces in $\AG(n,q)$, $n>2k-t$, pairwise meeting in at least a $t$-space. Let 
\begin{align*}
\psi(\mathcal{S}_a)=\min\{\dim(T)| T\subset \AG(n,q), \dim(T\cap \alpha)\geq t, \ \forall \alpha \in \mathcal{S}_a \}. \\
\end{align*}
Let  $\mathcal{T}$ be the set of all $\psi(\mathcal{S}_a)$-dimensional spaces in $\AG(n,q)$ that meet every element of $\mathcal{S}_a$ in at least a $t$-space. 
\begin{lemma}\label{remarksaffien}
	\begin{enumerate}
		\item $t\leq\psi(\mathcal{S}_a)\leq k$, and if $\psi(\mathcal{S}_a)=t$, then $\mathcal{S}_a$ is a $t$-pencil. 
		\item Let $T\in \mathcal{T}$, then all $k$-spaces through $T$ are contained in $\mathcal{S}_a$.
		\item The elements of $\mathcal{T}$ are $t$-intersecting in $\AG(n,q)$.
	\end{enumerate}
\end{lemma}
\begin{proof}Analogous to the proof of Lemma \ref{remarks}.
\end{proof}

\begin{lemma}\label{lemmacounting1affien}{
		Let $\psi(\mathcal{S}_a)=t+x,\  x\geq2$, $k>t+1$ and $n> 2k-t$. Then the number of elements of $\mathcal{S}_a$ through an affine $(t+x-j)$-space, with $j\in \{0,1,2,\dots, x\}$, is at most
		$(\theta_{k-t})^j \qbin{n-t-x}{k-t-x}.$ }
\end{lemma}
\begin{proof}
	Suppose that $\psi(\mathcal{S}_a)=t+x, \ x\geq 2$. 
	We prove, by induction on $j\in \{0,1,2,\dots, x\}$, that the number of $k$-spaces of ${\mathcal{S}_a}$  through an affine  $(t+x-j)$-space is at most $ \qbin{n-t-x}{k-t-x}(\theta_{k-t})^j$. Note that the statement is true for $j=0$, by Lemma \ref{remarksaffien}$(2)$.

	Let  $j\in \{1,2, 3,\dots, x\}$ and suppose now that the number of $k$-spaces of ${\mathcal{S}_a}$ through an affine $(t+x-j_0)$-space, is at most $(\theta_{k-t})^{j_0} \qbin{n-t-x}{k-t-x}$, for all $j_0<j$. Then we prove that this also holds for $j$.
	Consider an affine $(t+x-j)$-space $\gamma_j$. 
	Since $\psi(\mathcal{S}_p)=t+x$, we know that there exists a $k$-space $\pi_j$ of $\tilde{\mathcal{S}_a}$, meeting $\tilde{\gamma_j}$ in at most a $(t-1)$-space. 
	
	Suppose first that $\dim(\pi_j\cap \tilde{\gamma}_j)=t-1$ and let $\pi_{j\gamma}$ be a (projective) $(k-t)$-space in $\pi_j \setminus \tilde{\gamma}_j$. Then every element of $\tilde{\mathcal{S}_a}$ through $\tilde{\gamma}_j$ contains at least a point of $\pi_{j\gamma}$. Since there are $\theta_{k-t}$ points in $\pi_{j\gamma}$, and since the number of affine $k$-spaces in $\mathcal{S}_a$ through a $(t+x-j+1)$-space is at most $(\theta_{k-t})^{j-1}\qbin{n-t-x}{k-t-x}$, we find that the number of elements of $\tilde{\mathcal{S}_a}$ through $\tilde{\gamma}_j$ is at most $ (\theta_{k-t})^{j} \qbin{n-t-x}{k-t-x}.$

	Suppose now that every element $\pi$ of $\tilde{\mathcal{S}_a}$ meets $\tilde{\gamma}_j$ in a $t$-space or in at most a $(t-2)$-space. 
	Let $\max\{\dim(\tilde{\gamma}_j\cap \pi)|\pi \in \tilde{\mathcal{S}_a}, \dim(\tilde{\gamma}_j\cap \pi)< t \}=t-l$, then $l\geq 2$,  and suppose that $\pi_j\in \tilde{\mathcal{S}_a}$ is an element such that $\dim(\pi_j\cap \tilde{\gamma}_j)=t-l$. Let $\pi_{j\gamma}$ be a projective $(k-t+l-1)$-space in $\pi_j \setminus \tilde{\gamma}_j$.  Then every element of $\tilde{\mathcal{S}_a}$
	through $\tilde{\gamma}_j$ contains at least an $(l-1)$-space of $\pi_{j\gamma}$. Since there are $\qbin{k-t+l}{l}$ $(l-1)$-spaces in $\pi_{j\gamma}$, and since the number of affine $k$-spaces in $\mathcal{S}_a$ through a $(t+x-j+l)$-space is at most $(\theta_{k-t})^{j-l}\qbin{n-t-x}{k-t-x}$, we find that the number of elements of $\tilde{\mathcal{S}_a}$ through $\tilde{\gamma}_j$ is at most 
	$ \qbin{k-t+l}{l} (\theta_{k-t})^{j-l}\qbin{n-t-x}{k-t-x}$. Note that 
	\begin{align*}
	\qbin{k-t+l}{l} (\theta_{k-t})^{j-l}= \frac{(q^{k-t+l}-1)\dots (q^{k-t+1}-1)}{(q^l-1)\dots (q-1)} (\theta_{k-t})^{j-l} \leq  \left(\frac{(q^{k-t+1}-1)}{(q-1)}\right)^l (\theta_{k-t})^{j-l}=(\theta_{k-t})^{j}
	\end{align*}
	Hence, also in this case, 
	we find that the number of elements of $\tilde{\mathcal{S}_a}$ trough $\tilde{\gamma}_j$, and so, the number of elements of $\mathcal{S}_a$ through $\gamma_j$ is at most $ (\theta_{k-t})^{j} \qbin{n-t-x}{k-t-x}.$

\end{proof}

\begin{lemma}\label{lemmacountingaffien}
	Let $\mathcal{S}_a$ be a set of $k$-spaces, pairwise intersecting in at least a $t$-space in $\AG(n,q)$. If $\psi(\mathcal{S}_a)=t+x,\  x\geq2$, $k>t+1$ and $n> 2k-t$, then 
	$|\mathcal{S}_a|\leq q^x \qbin{t+x}{x}(\theta_{k-t})^{x} \qbin{n-t-x}{k-t-x}.$ 
\end{lemma}
\begin{proof}
	Suppose that $\psi(\mathcal{S}_a)=t+x, \ x\geq 2$. By Lemma \ref{lemmacounting1affien}, we know, for $j\in \{0,1,2,\dots, x\}$, that the number of $k$-spaces of ${\mathcal{S}_a}$  through an affine $(t+x-j)$-space is at most $ \qbin{n-t-x}{k-t-x}(\theta_{k-t})^j$. 
	
	Consider now an element $T\in \mathcal{T}$. Then every element of $\mathcal{S}_a$ meets $T$ in at least a $t$-space. Since there are at least $q^x \qbin{t+x}{x}$ affine $t$-spaces in $T$ and since every $t$-space is contained in at most $(\theta_{k-t})^{x} \qbin{n-t-x}{k-t-x}$ elements of 
	$\mathcal{S}_a$, we find that $\mathcal{S}_a$ has at most $q^x \qbin{t+x}{x}(\theta_{k-t})^{x} \qbin{n-t-x}{k-t-x}$ elements.
\end{proof}

\begin{lemma}\label{lemma1ellisaffien}
	Let $\mathcal{S}_a$ be a maximal set of $k$-spaces, pairwise intersecting in at least a $t$-space in $\AG(n,q)$, {$n> 2k-t, k>t$}. If $\psi(\mathcal{S}_a)=t+1$ and $|\mathcal{T}|\leq 2$, then \[|\mathcal{S}_a|\leq 2\qbin{n-t-1}{k-t-1}+ (\theta_{t+1}\theta_{k-t}-\theta_{t+1}-\theta_{k-t}) \theta_{k-t} \qbin{n-t-2}{k-t-2}.\]
\end{lemma}
\begin{proof}
	Let $T$ be an element of $\mathcal{T}$. Since $\mathcal{S}_a$ is a maximal set, we know that all $\qbin{n-t-1}{k-t-1}$ $k$-spaces through $T$ are contained in $\mathcal{S}_a$. Now we count the size of the set $\mathcal{S}_{a0}$ of $k$-spaces of $\mathcal{S}_a$ not through $T$. For every $\pi\in \mathcal{S}_{a0}$, $\dim(\pi\cap T)=t$, and let $E$ be an affine $t$-space in $T$. Then there exists an element $\alpha\in \mathcal{S}_{a0}$ not through $E$, and so $\dim(\alpha \cap E)=t-1$. Hence, every element $\pi$ of $\mathcal{S}_{a0}$, through $E$ must contain a $(t+1)$-space $\tau$, different from  $T$, such that $E\subset \tau$ and $\tau\cap (\alpha \setminus E) \neq \emptyset$. Note that there are $\theta_{k-t}-1$ possibilities for $\tau$. Fix such a $(t+1)$-space $\tau$. 
	\begin{itemize}
		\item If $\mathcal{T}=\{T\}$, we know that $\tau \notin \mathcal{T}$, and hence there exists an element $\sigma$ of $\mathcal{S}_a$, meeting $\tau$ in at most a $(t-1)$-space. Hence, every element of $\mathcal{S}_{a0}$ through $\tau$ meets $\sigma \setminus \tau$, and so the number of elements of $\mathcal{S}_{a0}$ through $\tau$ is at most $\theta_{k-t} \qbin{n-t-2}{k-t-2}$. 
		Since there are $q\theta_{t}$ possibilities for $E$, and at most $\theta_{k-t} -1$ for $\tau$, we have that \[|\mathcal{S}_a|\leq \qbin{n-t-1}{k-t-1}+q\theta_{t} (\theta_{k-t}-1)\theta_{k-t} \qbin{n-t-2}{k-t-2}.\]
		
		\item Suppose $|\mathcal{T}|=2$, and let $\mathcal{T}=\{T, \Psi\}$. If $\tau=\Psi$, then $\mathcal{S}_a$ contains all $\qbin{n-t-1}{k-t-1}$ $k$-spaces through $\tau$. If $\tau\neq \Psi$, then we can follow the argument in the previous item, and we find that the number of elements of $\mathcal{S}_{a0}$ through $\tau$ is at most $\theta_{k-t} \qbin{n-t-2}{k-t-2}$. 
		Note that there are $q\theta_{t}-1$ possibilities for $E\neq T\cap \Psi$, and at most $\theta_{k-t} -1$ for $\tau\neq \Psi, T$, through $E\neq T\cap \Psi$. Moreover, there are at most $\theta_{k-t} -2$ possibilities for $\tau\neq \Psi$ through $E= T\cap \Psi$. Hence, we have that \begin{align*}
		|\mathcal{S}_a|\leq& \qbin{n-t-1}{k-t-1}+ \sum_{E\subset T}\sum_{\tau \supseteq E} |\{\pi \in \mathcal{S}_{a0}|\tau \subset \pi\}| \\
		\leq & \qbin{n-t-1}{k-t-1}+ \sum_{E\neq T\cap \Psi }\sum_{\tau \supseteq E} \theta_{k-t} \qbin{n-t-2}{k-t-2}+\sum_{\tau \supset T\cap \Psi} |\{\pi \in \mathcal{S}_{a0}|\tau \subset \pi\}|\\
		\leq & \qbin{n-t-1}{k-t-1}+ (q\theta_{t}-1)(\theta_{k-t}-1) \theta_{k-t} \qbin{n-t-2}{k-t-2}+\sum_{\tau \neq  \Psi} \theta_{k-t}\qbin{n-t-2}{k-t-2}+\qbin{n-t-1}{k-t-1}\\
		\leq &  2\qbin{n-t-1}{k-t-1}+ (q\theta_{t}-1)(\theta_{k-t}-1) \theta_{k-t} \qbin{n-t-2}{k-t-2}+(\theta_{k-t}-2) \theta_{k-t}\qbin{n-t-2}{k-t-2}\\
		= &  2\qbin{n-t-1}{k-t-1}+ (\theta_{t+1}\theta_{k-t}-\theta_{t+1}-\theta_{k-t}) \theta_{k-t} \qbin{n-t-2}{k-t-2}.\\
		\end{align*}
	\end{itemize}
	The lemma follows since \[2\qbin{n-t-1}{k-t-1}+ (\theta_{t+1}\theta_{k-t}-\theta_{t+1}-\theta_{k-t}) \theta_{k-t} \qbin{n-t-2}{k-t-2}\geq \qbin{n-t-1}{k-t-1}+q\theta_{t} (\theta_{k-t}-1)\theta_{k-t} \qbin{n-t-2}{k-t-2}\] for $k>t, n\geq 2k-t, q\geq 2$ (see Lemma \ref{Lemma 47B}).
\end{proof}


From now on, we define $f_a(q,n,k,t)$ as the maximum of the number of elements in the sets described in Example \ref{example} and Example \ref{example2} for $n>2k-t$.
\begin{align*}
f_a(q,n,k,t)=\max\{&\theta_{k}-\theta_{k-t}+\qbin{n-t}{k-t}- q^{(k-t+1)(k-t)}\qbin{n-k-1}{k-t} ,\\ &\theta_{t+1}\cdot\left(\qbin{n-t-1}{k-t-1}-\qbin{n-t-2}{k-t-2}\right)+\qbin{n-t-2}{k-t-2} \}.
\end{align*}

From Lemma \ref{appendixaffienverschil1}, \ref{appendixaffienverschil2} and \ref{appendixaffienverschil3}, we find for $n>2k-t, k>t+1, q\geq 3$ that 

\begin{align*}
f_a(q,n,k,t)=&\begin{cases}
\theta_{k}-\theta_{k-t}+\qbin{n-t}{k-t}- q^{(k-t+1)(k-t)}\qbin{n-k-1}{k-t} &\text{if } k>2t+1 \\
\theta_{t+1}\cdot\left(\qbin{n-t-1}{k-t-1}-\qbin{n-t-2}{k-t-2}\right)+\qbin{n-t-2}{k-t-2} &\text{if } k\leq2t+1.
\end{cases}	
\end{align*}


\begin{theorem}
	Let $\mathcal{S}_a$ be a maximal set of $k$-spaces, pairwise intersecting in at least a $t$-space in $\AG(n,q)$, {$ k>t+1$, $t>0$,} with $q\geq 4$, and $n> 2k+t+2,$ (or $q=3$ and $n> 2k+t+3$). If $\mathcal{S}_a$ is not a $t$-pencil, then 
	\begin{align*}
	|\mathcal{S}_a|\leq f_a(q,n,k,t). 
	\end{align*}
	Equality occurs if and only if $\mathcal{S}_a$ is Example \ref{example} for $k>2t+1$ or Example \ref{example2} for $k\leq 2t+1$.
\end{theorem}
\begin{proof} 
	Let $\mathcal{S}_a$ be a maximal set of $k$-spaces, pairwise intersecting in at least a $t$-space, in $\AG(n,q)$, with $\mathcal{S}_a$ not a $t$-pencil, and suppose that $|\mathcal{S}_a|\geq f_a(q,n,k,t)$. From Lemma  \ref{lemmacountingaffien} and Lemma \ref{lemmaappendixaffienlelijk}, it follows that $\psi(\mathcal{S}_a)<t+2$ for  {$ k>t+1$, $t>0$,} with $q\geq 4$, and $n> 2k+t+2,$ (or $q=3$ and $n> 2k+t+3$). Since $\mathcal{S}_a$ is not a $t$-pencil, we find that $\psi(\mathcal{S}_a)>t$, and so $\psi(\mathcal{S}_a)=t+1$.
	
	From Lemma \ref{lemma1ellisaffien}, it follows that if $|{\mathcal{T}}|\leq 2$, then $|{{\mathcal{S}_a}}|\leq 2\qbin{n-t-1}{k-t-1}+ (\theta_{t+1}\theta_{k-t}-\theta_{t+1}-\theta_{k-t}) \theta_{k-t} \qbin{n-t-2}{k-t-2}$,
	a contradiction by Lemma \ref{appendixaffienbla} for $n>2k+t+2, k>t+1, q\geq 3, t>0$. 
	Hence, $|{{\mathcal{T}}}|>2$. 
	From Lemma \ref{remarksaffien}$(3)$, it follows that $\mathcal{T}$ is a $t$-intersecting set of $(t+1)$-spaces. Hence,
	$\tilde{\mathcal{T}}$, (and so $\mathcal{T}$), is contained in a $t$-pencil or all elements of $\tilde{\mathcal{T}}$, (and so $\mathcal{T}$),  are contained in a $(t+2)$-space (see Theorem \ref{basisEKRthm}).
	
	We first suppose that there is no $t$-space contained in all elements of  $\mathcal{T}$. Hence, we know that all elements of $\mathcal{T}$ are contained in a $(t+2)$-space $\Gamma$.
	We also know that the elements of $\mathcal{T}$ are $t$-intersecting in the affine space, and so, every $t$-space in $\tilde{\Gamma}\cap H_\infty$ is contained in at most one element of $\mathcal{T}$. Moreover, we also find that every element $\pi_1$ of $\mathcal{S}_a$ must meet $\Gamma$ in at least a $(t+1)$-space. This follows since $\pi_1$ must meet all elements of $\mathcal{T}$ in at least a $t$-space, and the elements of $\mathcal{T}$ are contained in a $(t+2)$-space.

	In this case, we claim that the number of elements of $\mathcal{S}_a$ is at most $\theta_{t+1}\cdot\left(\qbin{n-t-1}{k-t-1}-\qbin{n-t-2}{k-t-2}\right)+\qbin{n-t-2}{k-t-2}$. We also claim that equality holds 
	if and only if there exists a set $\mathcal{R}$ of affine $(t+1)$-spaces in $\Gamma$ such that no two elements of $\tilde{\mathcal{R}}$ meet in a $t$-space in $H_\infty$. Then, $\mathcal{S}_a$ is the set of all $k$-spaces that meet $\Gamma$ in an element of $\mathcal{R}$, and so, $\mathcal{S}_a$ is Example \ref{example2} . $(*)$
	
	To prove this, we first of all note
	that all $k$-spaces through $\Gamma$ are contained in $\mathcal{S}_a$. Consider a $t$-space $\alpha_t\subset \tilde{\Gamma}\cap H_\infty$. Then we count the number of elements of $\tilde{\mathcal{S}_a}$ through $\alpha_t$, not through $\Gamma$. There are two possibilities.
	\begin{itemize}
		\item All these elements meet $\tilde{\Gamma}$ in the same affine $(t+1)$-space through $\alpha_t$. Then the number of elements of $\tilde{\mathcal{S}_a}$ through $\alpha_t$ and not through $\Gamma$ is $\qbin{n-t-1}{k-t-1}-\qbin{n-t-2}{k-t-2}$. If this is the case for all $t$-spaces $\alpha_t \subset \tilde{\Gamma}\cap H_\infty$, then $|\mathcal{S}_a|=\theta_{t+1}\cdot\left(\qbin{n-t-1}{k-t-1}-\qbin{n-t-2}{k-t-2}\right)+\qbin{n-t-2}{k-t-2}$, and $\mathcal{S}_a$ is of the described form. 
		\item  There is a $t$-space $\alpha_t\subset \tilde{\Gamma}\cap H_\infty$, such that there are two elements $\pi_1, \pi_2\in \mathcal{S}_a$, not contained in $\Gamma$, with $\alpha_t\in \tilde{\pi}_1, \tilde{\pi}_2$, but $\pi_1\cap \Gamma\neq \pi_2\cap \Gamma$. Then every element $\pi$ of $\tilde{\mathcal{S}_a}$ through $\alpha_t$, not through $\pi_1\cap \Gamma$, meets $\tilde{\Gamma}$ in a $(t+1)$-space through $\alpha_t$ and meets $\pi_1$ in an affine point outside of $\Gamma$. For the elements of $\tilde{\mathcal{S}_a}$ not through $\tilde{\Gamma}$ and through $\pi_1\cap \Gamma$, we can use the same argument by using $\pi_2$.
		  Note that $\tilde{\pi}$ meets $\tilde{\Gamma}$ in one of the $q$ affine $(t+1)$-spaces in $\tilde{\Gamma}$ through  $\alpha_t$. 
		  	\begin{itemize}
		  		\item If $\tilde{\pi} \cap \tilde{\Gamma}\neq\tilde{\pi}_1 \cap \tilde{\Gamma}$ then there are $q^{k-t-1}-1$ ways to extend this $(t+1)$-space $\pi \cap \Gamma$ to a $(t+2)$-space, meeting $\pi_1$ in an affine $(t+1)$-space, not in $\Gamma$. By investigating the quotient space of $\alpha_t$ in $\tilde{\pi}_1$: there are $q^{k-t-1}$ ways to extend $\tilde{\pi} \cap \tilde{\Gamma}$ to a $(t+2)$-space meeting $\pi_1$ in an affine $(t+1)$-space, and one of these extended $(t+2)$-spaces is equal to $\Gamma$.
		  		\item If $\tilde{\pi} \cap \tilde{\Gamma}=\tilde{\pi}_1 \cap \tilde{\Gamma}$ then $\tilde{\pi} \cap \tilde{\Gamma}\neq\tilde{\pi}_2 \cap \tilde{\Gamma}$. Hence, we can use the same argument from the previous point to see that there are $q^{k-t-1}-1$ ways to extend this $(t+1)$-space to a $(t+2)$-space, meeting $\pi_2$ in an affine $(t+1)$-space, not in $\Gamma$.
		  	\end{itemize}
		  	Hence, there are at most $q (q^{k-t-1}-1)\cdot  \qbin{n-t-2}{k-t-2}$ elements of $\tilde{\mathcal{S}_a}$ through $\alpha_t$ and not through $\Gamma$. The fact that 
		  	\begin{align*}
		  		q(q^{k-t-1}-1)=\frac{q^{2k-2t-1}-2q^{k-t}+q}{q^{k-t-1}-1}<\frac{q^{n-t-1}-q^{k-t-1}-2q^{k-t}+q}{q^{k-t-1}-1}<\frac{q^{n-t-1}-q^{k-t-1}}{q^{k-t-1}-1},
		  	\end{align*}for $n>2k-t, k>t+1, q\geq 3$, implies that 
		  	$q (q^{k-t-1}-1)\cdot  \qbin{n-t-2}{k-t-2}<\frac{q^{n-t-1}-q^{k-t-1}}{q^{k-t-1}-1}\qbin{n-t-2}{k-t-2}= \qbin{n-t-1}{k-t-1}-\qbin{n-t-2}{k-t-2}$. This
		  	proves the claim.
	\end{itemize}
	
	So $|\mathcal{S}_a|= \theta_{t+1}\cdot\left(\qbin{n-t-1}{k-t-1}-\qbin{n-t-2}{k-t-2}\right)+\qbin{n-t-2}{k-t-2} $ if there is no $t$-space contained in all elements of $\mathcal{T}$.  This number is larger than $\theta_{k}-\theta_{k-t}+\qbin{n-t}{k-t}- q^{(k-t-1)(k-t)}\qbin{n-k-1}{k-t}$, if and only if $k\leq2t+1$. So, for $k>2t+1$, we find a contradiction. Hence, for $k>2t+1$, we know that the elements of $\mathcal{T}$ are contained in a $t$-pencil with vertex the affine $t$-space $\delta$. Let $Z$ be the span of all elements of $\mathcal{T}$ and let $\dim(Z)=t+x$, $x\geq 2$. Since $\mathcal{S}_a$ is not a $t$-pencil, we know that there are $k$-spaces in $\mathcal{S}_a$ that do not contain $\delta$.  Note that these elements of $\mathcal{S}_a$, not through $\delta$, meet $\delta$ in a $(t-1)$-space, and meet $Z$ in a $(t+x-1)$-space. The dimension of the span $Z$ of all the $(t+1)$-spaces in $\mathcal{T}$ is at most $k+1$: if $\dim(Z)>k+1$, then every $k$-space of $\mathcal{S}_a$, not through $\delta$ would meet $Z$ in a subspace with dimendion $\dim(Z)-1>k$, a contradiction. 
	
	Let $\pi\in \mathcal{S}_a$ be an element that does not contain $\delta$. Note that every element of $\mathcal{S}_a$ through $\delta$ has at least a $(t+1)$-space in common with $\langle \pi, \delta \rangle$.
	Now we claim that all elements of $\mathcal{S}_a$, not through $\delta$, are contained in $\langle \pi, \delta \rangle$.
	Suppose that this is not the case, then there exists an element $\pi_2\in \mathcal{S}_a$ with $\delta \nsubseteq \pi_2$ and $\pi_2 \nsubseteq\langle \pi,\delta \rangle$. Then every element of $\mathcal{S}_a$ through $\delta$ meets both $\pi\setminus \delta$ and $\pi_2\setminus \delta$. Hence the number of elements of $\mathcal{S}_a$, through $\delta$, is at most $\theta_{k-t}^{2}\qbin{n-t-2}{k-t-2}+\theta_{k-t-1}\qbin{n-t-1}{k-t-1}$. Here, the first term is an upper bound on the number of elements meeting both $\pi \setminus \pi_2$ and $\pi_2 \setminus \pi$. The second term is an upper bound on the number of elements meeting $(\pi\cap \pi_2) \setminus \delta$, since $\dim( (\pi\cap \pi_2)\setminus \delta)\leq k-t-1$. Every element of $\mathcal{S}_a$ not through $\delta$ meets $Z$ in a $(t+x-1)$-space. This implies that that $|\mathcal{S}_a|\leq \theta_{t+x}\qbin{n-t-x+1}{k-t-x+1}+\theta_{k-t}^{2}\qbin{n-t-2}{k-t-2}+\theta_{k-t-1}\qbin{n-t-1}{k-t-1}$. For $n>2k+t+2, k>2t+1, t>0, x\geq 3, q\geq 3$; this gives a contradiction by Lemma \ref{appendixlaatsteaffienextra}, since $|S|\geq f_a(q,n,k,t)$. Now, in a last step, we also have to find a contradiction for $x=2$, and so $Z$ a $(t+2)$-space. In this situation, all $k$-spaces not through $\delta$ must meet $Z$ in a $(t+1)$-space, not through $\delta$. Now, every two elements of $\mathcal{S}$, not through $\delta$ must meet in at least a $t$-space. The same argument, used to deduce $(*)$, can be used to show the following. For every $t$-space $\alpha_t \subset \tilde{Z}\cap H_\infty$, $\tilde{\delta}\cap H_\infty \nsubseteq\alpha_t$, we have that all elements of $\tilde{\mathcal{S}_a}$ through $\alpha_t$ must meet $Z$ in the same $(t+1)$-space. Hence, there are at most $\theta_{t+1}-\theta_1$ possibilities for the intersection $\pi\cap Z$, with $\pi\in \mathcal{S}_a$, $\delta \nsubseteq \pi$, and there are at most $\qbin{n-t-1}{k-t-1}$ $k$-spaces through a fixed $(t+1)$-space. Hence we find that the number of elemtents of $\mathcal{S}_a$, not through $\delta$ is at most $q^2 \theta_{t-1}\qbin{n-t-1}{k-t-1}$, and so $|\mathcal{S}_a|\leq q^2 \theta_{t-1}\qbin{n-t-1}{k-t-1}+\theta_{k-t}^{2}\qbin{n-t-2}{k-t-2}+\theta_{k-t-1}\qbin{n-t-1}{k-t-1}$. This gives a contradiction for $n>2k+t+2, k>2t+1$ and $q\geq 3$ by Lemma \ref{appendixlaatsteaffien} since $|S|\geq f_a(q,n,k,t)$.
	Hence, we find that every element of $\mathcal{S}_a$, not through $\delta$ is contained in $\langle \delta, \pi \rangle$, and so $\mathcal{S}_a$ is contained in Example \ref{example}. The theorem follows from the maximality of $\mathcal{S}_a$.
	
\end{proof}

\section{Erd\H{o}s-Ko-Rado result in $\AG(n,q)$}\label{sectionappendix}
In \cite{GuoXu}, the authors prove that the largest $t$-intersecting set of $k$-spaces in $\AG(n,q)$, with $n\geq 2k+t+2$, is the set of all $k$-spaces through a fixed affine $t$-space. They use geometrical and combinatorical techniques, but they do not use the connection between $\AG(n,q)$ and $\PG(n,q)$. Below we give a shorter proof for this result, for $n\geq 2k+1$, by using the corresponding result in $\PG(n,q)$.
\begin{theorem}{\textbf{\cite[Theorem 1]{frankl}}}\label{q-analogue}
	Let $t$ and $k$ be integers, with $0 \leq t \leq k$. Let $\mathcal{S}$ be a set of $k$-spaces
	in $\PG(n,q)$, pairwise intersecting in at least a $t$-space.
	\begin{itemize}
		\item [(i)] If $n \geq 2k + 1$, then $| \mathcal{S} |
		\leq \qbin{n-t}{k-t}$. Equality holds if and only if $\mathcal{S}$ is the set of all the $k$-spaces, containing a fixed $t$-space of $\PG(n,q)$, or $n = 2k +1$ and $\mathcal{S}$ is the set of all the $k$-spaces in a fixed $(2k -t)$-space.
		\item[(ii)] If $2k-t \leq n \leq 2k$, then $|\mathcal{S}| \leq \qbin{2k-t+1}{k-t}$. Equality holds if and only if $\mathcal{S}$ is the set of all the $k$-spaces in a fixed $(2k-t)$-space.
	\end{itemize}
\end{theorem}

\begin{theorem}
	Let $\mathcal{S}$ be a set of $k$-spaces in $\AG(n,q)$, $n\geq2k+1, k\geq t \geq 0$, pairwise
	intersecting in a $t$-space such that $|\mathcal{S}|$ is maximal, then $\mathcal{S}$ is a $t$-pencil.
\end{theorem}
\begin{proof}
	Note first that the set of all affine $k$-spaces through a fixed affine $t$-space is a set of $k$-spaces pairwise intersecting in at least a $t$-space with size $\qbin{n-t}{k-t}$. Suppose now that there exists a $t$-intersecting set $\mathcal{S}$ of $k$-spaces in $AG(n,q)$ with at least $\qbin{n-t}{k-t}$ elements, which is not a $t$-pencil. Every affine element $\alpha$ in $\mathcal{S}$ can be extended to the corresponding projective $k$-space $\tilde{\alpha}$ in $\PG(n,q)$. Let $\tilde{\mathcal{S}}$ be the set of these extended $k$-spaces. Note that $\tilde{\mathcal{S}}$ is a $t$-intersecting set of $k$-spaces in $\PG(n,q)$ with $|\tilde{\mathcal{S}}|\geq\qbin{n-t}{k-t}$. It follows from Theorem \ref{q-analogue} that $n=2k+1$ and that all elements of $\tilde{\mathcal{S}}$ are contained in a projective $(2k-t)$-space. Since the number of affine spaces in a projective $(2k-t)$-space is $\qbin{2k-t+1}{k+1}-\qbin{2k-t}{k+1}$, we see that in this case $|\mathcal{S}|<\qbin{n-t}{k-t}$. Hence an affine $t$-pencil is the only example of a set of pairwise $t$-intersecting $k$-spaces in $\AG(n,q)$ with size at least $\qbin{n-t}{k-t}$.
\end{proof}

\appendix
\section{Appendix}
We start with some bounds on the binomial Gaussian coefficient that will be usefull to prove the inequalities used during the proofs in this article. For the proofs of these bounds we refer to \cite[Lemma 2.1, Lemma 2.2]{thesisferdinand}.

\begin{lemma}\label{bounds}
	Let $n\geq k \geq 0$.
	\begin{enumerate}
	\item Let $q\geq 3$. Then $\qbin{n}{k} \leq 2q^{k(n-k)}.$
	\item Let $q\geq 4$. Then $\qbin nk \leq \left( 1+\frac{2}{q}\right) q^{k(n-k)}.$
	\item Let $q\geq 2$ and $n\geq 1$. Then $\theta_n \leq \frac{q^{n+1}}{q-1}.$
	\end{enumerate}
Let $n>k>0$, then $\qbin{n}{k} \geq \left( 1+\frac{1}{q}\right) q^{k(n-k)}.$
\end{lemma}

\begin{lemma}\label{Lemma 47}
	For $n>2k-t$, $k>t$ and $q\geq 2$ it holds that
	\begin{align*}
		 2\qbin{n-t-1}{k-t-1}+ (\theta_{t+1}\theta_{k-t}-\theta_{t+1}-1) \theta_{k-t} \qbin{n-t-2}{k-t-2}\geq \qbin{n-t-1}{k-t-1}+\theta_{t+1} (\theta_{k-t}-1)\theta_{k-t} \qbin{n-t-2}{k-t-2}. 
	\end{align*}\end{lemma}
	\begin{proof}
	The inequality is equivalent to
	\begin{align*}
		& \qbin{n-t-1}{k-t-1}+ (\theta_{t+1}\theta_{k-t}-\theta_{t+1}-1) \theta_{k-t} \qbin{n-t-2}{k-t-2}\geq \theta_{t+1} (\theta_{k-t}-1)\theta_{k-t} \qbin{n-t-2}{k-t-2}\\
		\Leftrightarrow & \  \qbin{n-t-1}{k-t-1} \geq \theta_{k-t}\qbin{n-t-2}{k-t-2} \\
		\Leftrightarrow & \ \frac{q^{n-t-1}-1}{q^{k-t-1}-1}\geq \frac{q^{k-t+1}-1}{q-1}\\
		\Leftrightarrow & \ q^{n-t}-q^{n-t-1}-q+1\geq q^{2k-2t}-q^{k-t+1}-q^{k-t-1}+1\\
		\Leftrightarrow & \ \left(q^{n-t}-q^{n-t-1}-q^{2k-2t}\right)+\left(q^{k-t+1}-q\right)+q^{k-t-1} \geq 0
	\end{align*}
	The last inequality holds since all terms in the left hand side of the last inequality are non negative for $n>2k-t$, $k>t$ and $q\geq 2$.
\end{proof}

\begin{lemma}\label{Lemma 47B}
For $n\geq 2k-t, k>t$ and $q\geq 2$ it holds that
\begin{align*}
2\qbin{n-t-1}{k-t-1}+ (\theta_{t+1}\theta_{k-t}-\theta_{t+1}-\theta_{k-t}) \theta_{k-t} \qbin{n-t-2}{k-t-2}\geq \qbin{n-t-1}{k-t-1}+q\theta_{t} (\theta_{k-t}-1)\theta_{k-t} \qbin{n-t-2}{k-t-2}.
\end{align*}
  \end{lemma}
\begin{proof}
	The inequality is equivalent to
	\begin{align*}
	& \qbin{n-t-1}{k-t-1}+ (\theta_{t+1}\theta_{k-t}-\theta_{t+1}-\theta_{k-t}) \theta_{k-t} \qbin{n-t-2}{k-t-2}\geq q\theta_{t} (\theta_{k-t}-1)\theta_{k-t} \qbin{n-t-2}{k-t-2}\\
	\Leftrightarrow & \  \qbin{n-t-1}{k-t-1} \geq \theta_{k-t}\qbin{n-t-2}{k-t-2} \\
	\Leftrightarrow & \ \frac{q^{n-t-1}-1}{q^{k-t-1}-1}\geq \frac{q^{k-t+1}-1}{q-1}\\
	\Leftrightarrow & \ q^{n-t}-q^{n-t-1}-q+1\geq q^{2k-2t}-q^{k-t+1}-q^{k-t-1}+1\\
	\Leftrightarrow & \ \left(q^{n-t}-q^{n-t-1}-q^{2k-2t}\right)+\left(q^{k-t+1}-q\right)+q^{k-t-1} \geq 0
	\end{align*}
	The last inequality holds since all terms in the left hand side of the last inequality are non negative for $n>2k-t$, $k>t$ and $q\geq 2$.
\end{proof}

\begin{lemma}\label{appendixprojverschil1}
	Let $n>2k-t, q\geq 3$ and consider Example \ref{examplep} and Example \ref{example2p} in $\PG(n,q)$. The number of elements in Example \ref{examplep} is larger than the number of elements in Example \ref{example2p}  if $k>2t+2$.
\end{lemma}
\begin{proof}
	Let $S_{2.1}$ and $S_{2.3}$ be the set of elements in Example \ref{examplep} and in Example \ref{example2p} respectively.
	Suppose that $k>2t+2$. Then we have to prove that $|S_{2.3}|<|S_{2.1}|$. Suppose to the contrary that $|S_{2.3}|\geq |S_{2.1}|$. Then 
	\begin{align*}
		&\qbin{n-t-2}{k-t-2} \left( 1+ \theta_{t+2}q^{k-t-1}\frac{q^{n-k}-1}{q^{k-t-1}-1}  \right) \geq \theta_{k+1}+\sum_{j=0}^{k-t-2} \qbin{k-t+1}{j+1}q^{(k-t-j)(k-t-j-1)}\qbin{n-k-1}{k-t-j-1}\\
		&\xRightarrow{j=0} \qbin{n-t-2}{k-t-2} \left( 1+ \theta_{t+2}q^{k-t-1}\frac{q^{n-k}-1}{q^{k-t-1}-1}  \right) > \theta_{k-t}q^{(k-t)(k-t-1)}\qbin{n-k-1}{k-t-1}\\
		&\xRightarrow{L. \ref{bounds}} 2q^{(k-t-2)(n-k)}\left( 1+ \theta_{t+2}q^{k-t-1}\frac{q^{n-k}-1}{q^{k-t-1}-1}  \right) > \left(1+\frac{1}{q}\right)^2 q^{k-t+(k-t)(k-t-1)+(k-t-1)(n-2k+t)}\\
		&\Rightarrow 2+ 2 \theta_{t+2} q^{k-t-1}\frac{q^{n-k}-1}{q^{k-t-1}-1}>\left(1+\frac{1}{q}\right)^2 q^{n-t}\\
		& \Rightarrow 2(q-1)(q^{k-t-1}-1)+ 2(q^{t+3}-1)q^{k-t-1}(q^{n-k}-1)>\left(1+\frac{1}{q}\right)^2 (q-1)(q^{k-t-1}-1)q^{n-t}\\
		&\Rightarrow 2q^{k-t}-2q^{k-t-1}-2q+2+2q^{n+2}-2q^{n-t-1}-2q^{k+2}+2q^{k-t-1}\\ & \qquad>\left(1+\frac{1}{q}\right)^2(q^{n+k-2t}-q^{n+k-2t-1}-q^{n-t+1}+q^{n-t})\\
		&\qquad\geq q^{n+k-2t}+q^{n+k-2t-1}-q^{n+k-2t-2}-q^{n+k-2t-3}-q^{n-t+1}-q^{n-t}+q^{n-t-1}+q^{n-t-2}\\
		&\Rightarrow \left(  2q^{n+2}+ q^{n-t+1} - q^{n+k-2t} \right)  +\left( q^{n+k-2t-2}+q^{n+k-2t-3}+q^{n-t}-q^{n+k-2t-1}  \right)\\
		&\qquad +\left( 2q^{k-t}+2-2q^{k+2}  \right)+(-2q-3q^{n-t-1}-q^{n-t-2}) >0
					\end{align*}
					In the left hand side of the last inequality all terms are at most zero for $k>2t+2$ and $q\geq 3$. Hence we find a contradiction which proves the statement.
\end{proof}

\begin{lemma}\label{appendixprojverschil2}
	Let $n>2k-t, k>t+1, q\geq 3$ and consider Example \ref{examplep} and Example \ref{example2p} in $\PG(n,q)$. The number of elements in Example \ref{example2p} is larger than the number of elements in Example \ref{examplep}  if $k<2t+2$.
\end{lemma}
\begin{proof}
	Let $S_{2.1}$ and $S_{2.3}$ be the set of elements in Example \ref{examplep} and in Example \ref{example2p} respectively.
	Suppose that $k<2t+2$. Then we have to prove that $|S_{2.3}|>|S_{2.1}|$. Suppose to the contrary that $|S_{2.3}|\leq |S_{2.1}|$. Then 
	\begin{align*}
		\qbin{n-t-2}{k-t-2} \left( 1+ \theta_{t+2}q^{k-t-1}\frac{q^{n-k}-1}{q^{k-t-1}-1}  \right) &\leq \theta_{k+1}+\sum_{j=0}^{k-t-2} \qbin{k-t+1}{j+1}q^{(k-t-j)(k-t-j-1)}\qbin{n-k-1}{k-t-j-1}\\ & <q^{k-t+1}\theta_t+ \theta_{k-t}\qbin{n-t-1}{k-t-1}\\
	\end{align*}
	The last inequality follows since $q^{k-t+1}\theta_t$ is the number of elements of $S_{2.1}$ contained in $\langle \pi, \delta \rangle$ but not containing $\delta$. The second term $\theta_{k-t}\qbin{n-t-1}{k-t-1}$ is the number of $k$-spaces through a $(t+1)$-space in $\langle \pi, \delta \rangle$ through $\delta$.
	\begin{align*}
		&\xRightarrow{L. \ref{bounds}} \left(1+\frac 1q\right)q^{(n-k)(k-t-2)}\left( \theta_{t+2}q^{k-t-1}\frac{q^{n-k}-1}{q^{k-t-1}-1}  \right) < q^{k-t+1} \theta_t+2\theta_{k-t} q^{(n-k)(k-t-1)}\\
		& \Rightarrow \left(1+\frac 1q \right) (q^{t+3}-1)(q^{n-k}-1) q^{k-t-1} < \frac{ (q^{t+1}-1)(q^{k-t-1}-1)}{q^{(n-k)(k-t-2)-k+t-1}}+2(q^{k-t+1}-1)(q^{k-t-1}-1)q^{n-k}\\
		&\Rightarrow q^{n+2}+q^{k-t-1}+q^{n+1}+q^{k-t-2}-q^{n-t-1}-q^{k+2}-q^{n-t-2}-q^{k+1}\\ & \qquad \qquad <2q^{n+k-2t}+2q^{n-k}-2q^{n-t-1}-2q^{n-t+1}+{q^{k-t+1-(n-k)(k-t-2)} (q^{t+1}-1)(q^{k-t-1}-1)}\\
		& \Rightarrow\left( q^{n+2}-2q^{n+k-2t}-{q^{k-t+1-(n-k)(k-t-2)} (q^{t+1}-1)(q^{k-t-1}-1)}\right) +\left( q^{n+1}-q^{k+1}-q^{k+2} \right) \\ &\qquad \qquad + \left(  2q^{n-t-1}-q^{n-t-1}-q^{n-t-2}\right) +\left(2q^{n-t+1}+q^{k-t-1}-2q^{n-k} \right)+q^{k-t-2}<0
	\end{align*}
	Now, the contradiction follows since all terms in the left hand side of the last inequality are positive. For the last four terms this follows immediately since $n>2k-t,k< 2t+2, k>t+1, q\geq 3$.  We end this proof by proving that the first term is also positive. Since $k\geq t+2$ and $n>2k-t=k+(k-t)\geq k+1$ we have that
	\begin{align*}
	&0<(n-k-1)(k-t-1)\\
	&\Leftrightarrow  n+2 > 2k-t+1-(n-k)(k-t-2)\\
	&\Rightarrow	q^{n+2}>q^{2k-t+1-(n-k)(k-t-2)}>q^{k-t+1-(n-k)(k-t-2)} (q^{t+1}-1)(q^{k-t-1}-1).
	\end{align*}
\end{proof}

\begin{lemma}\label{appendixprojverschil3}
	Let $n>2k-t, q\geq 3$, and consider Example \ref{examplep} and Example \ref{example2p} in $\PG(n,q)$. The number of elements in Example \ref{example2p} is larger than the number of elements in Example \ref{examplep} if $k=2t+2$.
\end{lemma}
\begin{proof}
	Let $S_{2.1}$ and $S_{2.3}$ be the set of elements in Example \ref{examplep} and in Example \ref{example2p} respectively.
	Suppose that $k= 2t+2$, then we have to prove that $|S_{2.3}|\geq |S_{2.1}|$.
	Consider the set $S_{2.3}$, and let $\tau$ be a $t$-space in $\PG(n,q)$, disjoint from $\Gamma$. Then we find that $|S_{2.3}|=\qbin{n-t-2}{t}+\sum_{i=-1}^{t} R_i(n, 2t+2,t)$, with $R_i(n, k,t)=\{ \alpha\in S_{2.3}| \dim(\alpha\cap \Gamma)=t+1, \dim(\alpha\cap \tau)=i  \}$. Since there are $\theta_{t+2}$ $(t+1)$-spaces in $\Gamma$, and $\qbin{t+1}{i+1}$ $i$-spaces in $\tau$ we have, by using Lemma \ref{lemmadisjunct}, that
	\begin{align}
	|S_{2.3}|&=\qbin{n-t-2}{t}+\sum_{i=-1}^{t} \theta_{t+2} \qbin{t+1}{i+1}\left(q^{(t-i)^2}\qbin{n-2t-2}{t-i}-q^{(t-i)(t-i-1)}\qbin{n-2t-3}{t-i-1}\right)\nonumber \\
	&=\qbin{n-t-2}{t}+\sum_{i=-1}^{t-1} \theta_{t+2} \qbin{t+1}{i+1}q^{(t-i)(t-i-1)}\qbin{n-2t-3}{t-i-1}\frac{q^{n-t-i-2}-2q^{t-i}+1}{q^{t-i}-1}+\theta_{t+2}\nonumber \\
	&=\qbin{n-t-2}{t}+\sum_{j=0}^{t} \theta_{t+2} \qbin{t+1}{j}q^{(t+1-j)(t-j)}\qbin{n-2t-3}{t-j}\frac{q^{n-t-j-1}-2q^{t-j+1}+1}{q^{t-j+1}-1}+\theta_{t+2}.\label{vgl22}
	\end{align}
	On the other hand we have that
	\begin{align}\label{vgl32}
	|S_{2.1}|&=\theta_{2t+3}+\sum_{j=0}^{t} \qbin{t+3}{j+1}q^{(t+2-j)(t+1-j)}\qbin{n-2t-3}{t+1-j}.
	\end{align}
	From (\ref{vgl22}) and (\ref{vgl32}), it follows that $|S_{2.3}|-|S_{2.1}|$ is equal to 
	\begin{align}
\underbrace{\qbin{n-t-2}{t}+\theta_{t+2}-\theta_{2t+3}}_{=w_1} \nonumber+\sum_{j=0}^{t} q^{(t+1-j)(t-j)}\qbin{n-2t-3}{t-j} \qbin{t+1}{j}\frac{q^{t+3}-1}{q^{t-j+1}-1} w_2,
	\end{align}
	with 
	\begin{align*}
		w_2=\frac{q^{n-t-j-1}-2q^{t-j+1}+1}{q-1} -q^{2(t+1-j)} \frac{(q^{n-3t-3+j}-1)(q^{t+2}-1)}{(q^{j+1}-1)(q^{t+2-j}-1)}
	\end{align*}
	We will prove that $w_1\geq 0$ and $w_2\geq 0$, which proves that $|S_{2.3}|\geq |S_{2.1}|$ for $k= 2t+2$.
	\begin{align*}
	w_1=\qbin{n-t-2}{t}+\theta_{t+2}-\theta_{2t+3}&\overset{L. \ref{bounds}}{ \geq} \left(1+\frac{1}{q}\right) q^{(n-2t-2)t}+\theta_{t+2}-\frac{q^{2t+4}}{q-1}\\
	&\geq\frac{1}{q(q-1)}\left(  q^{(n-2t-2)t+2}-q^{(n-2t-2)t} -q^{2t+5} \right)+\theta_{t+2}
	\end{align*}
	It is sufficient to prove that  $q^{(n-2t-2)t+2}> q^{2t+5} $. This inequality holds for $n> 2t+4+\frac 3t$. For $t>1$ this assumption holds since $n>2k-t=3t+4$. If $t=1$ and $n> 9$ we also find that $q^{(n-2t-2)t+2}> q^{2t+5} $. For $n=9$ and $t=1$, we find that $w_1=\theta_3>0$. In the last remaining case; $n=8, t=1$, we have that $w_1<0$. For this case, we used a computer algebra packet to calculate both numbers $|S_{2.3}|, |S_{2.1}|$ to  see that $|S_{2.3}|\geq |S_{2.1}|$.
	\begin{align*}
	&w_2=\frac{q^{n-t-j-1}-2q^{t-j+1}+1}{q-1} -q^{2(t+1-j)} \frac{(q^{n-3t-3+j}-1)(q^{t+2}-1)}{(q^{j+1}-1)(q^{t+2-j}-1)}\\
	&  =\frac{\left(q^{n-j+1}+q^{n-t-j}-q^{n-t}-q^{n-2j+1}\right)+\left(q^{3t-2j+5}-2q^{2t-j+4}-q^{3t-2j+4}\right)}{(q-1)(q^{j+1}-1)(q^{t+2-j}-1)}\\
	& \hspace{3cm} + \frac{\left(2q^{t+2}-2q^{t-j+1}\right)+\left(q^{t+3}-q^{j+1}-q^{t-j+2}\right)+q^{2t-2j+2}+q^{2t-2j+3}+1}{(q-1)(q^{j+1}-1)(q^{t+2-j}-1)}
	\end{align*}
	For $0\leq j\leq t$, we find that both the nominator and denominator are positive, since we have that $q\geq 3$. So $w_2\geq 0$.
	Hence, we have that $|S_{2.3}|> |S_{2.1}|$.
\end{proof}

\begin{lemma}\label{appendixaffienverschil1}
	Let $n>2k-t, q\geq 3$ and consider Example \ref{example} and Example \ref{example2} in $\AG(n,q)$. The number of elements in Example \ref{example} is larger than the number of elements in Example \ref{example2}  if $k>2t+1$.
\end{lemma}
\begin{proof}
	Let $R_{3.1}$ and $R_{3.3}$ be the set of elements in Example \ref{example} and in Example \ref{example2} respectively.
	Suppose that $k>2t+1$. Then we have to prove that $|R_{3.3}|<|R_{3.1}|$. Suppose to the contrary that $|R_{3.3}|\geq |R_{3.1}|$. Then 
	\begin{align*}
	&\qbin{n-t-2}{k-t-2} \left( 1+ \theta_{t+1}q^{k-t-1}\frac{q^{n-k}-1}{q^{k-t-1}-1}  \right) \geq \theta_{k}+\sum_{j=0}^{k-t-2} \qbin{k-t+1}{j+1}q^{(k-t-j)(k-t-j-1)}\qbin{n-k-1}{k-t-j-1}\\
	&\xRightarrow{j=0} \qbin{n-t-2}{k-t-2} \left( 1+ \theta_{t+1}q^{k-t-1}\frac{q^{n-k}-1}{q^{k-t-1}-1}  \right) >  \theta_{k-t}q^{(k-t)(k-t-1)}\qbin{n-k-1}{k-t-1}\\
	&\xRightarrow{L. \ref{bounds}} 2q^{(k-t-2)(n-k)}\left( 1+ \theta_{t+1}q^{k-t-1}\frac{q^{n-k}-1}{q^{k-t-1}-1}  \right) > \left(1+\frac{1}{q}\right)^2 q^{k-t+(k-t)(k-t-1)+(k-t-1)(n-2k+t)}\\
	&\Rightarrow 2+ 2 \theta_{t+1} q^{k-t-1}\frac{q^{n-k}-1}{q^{k-t-1}-1}>\left(1+\frac{1}{q}\right)^2 q^{n-t}\\
	& \Rightarrow 2(q-1)(q^{k-t-1}-1)+ 2(q^{t+2}-1)q^{k-t-1}(q^{n-k}-1)>\left(1+\frac{1}{q}\right)^2 (q-1)(q^{k-t-1}-1)q^{n-t}\\
	&\Rightarrow 2q^{k-t}-2q^{k-t-1}-2q+2+2q^{n+1}-2q^{n-t-1}-2q^{k+1}+2q^{k-t-1}\\ & \qquad>\left(1+\frac{1}{q}\right)^2(q^{n+k-2t}-q^{n+k-2t-1}-q^{n-t+1}+q^{n-t})\\
	&\qquad\geq q^{n+k-2t}+q^{n+k-2t-1}-q^{n+k-2t-2}-q^{n+k-2t-3}-q^{n-t+1}-q^{n-t}+q^{n-t-1}+q^{n-t-2}\\
	&\Rightarrow \left(  2q^{n+1}+ q^{n-t+1} - q^{n+k-2t} \right)  +\left( q^{n+k-2t-2}+q^{n+k-2t-3}+q^{n-t}-q^{n+k-2t-1}  \right)\\
	&\qquad +\left( 2q^{k-t}+2-2q^{k+1}  \right)+(-2q-3q^{n-t-1}-q^{n-t-2}) >0
	\end{align*}
	In the left hand side of the last inequality all terms are at most zero for $k>2t+1$ and $q\geq 3$. Hence we find a contradiction which proves the statement.
\end{proof}

\begin{lemma}\label{appendixaffienverschil2}
	Let $n>2k-t, k>t+1,  q\geq 3$ and consider Example \ref{example} and Example \ref{example2} in $\AG(n,q)$. The number of elements in Example \ref{example2} is larger than the number of elements in Example \ref{example}  if $k<2t+1$.
\end{lemma}
\begin{proof}
	Let $R_{3.1}$ and $R_{3.3}$ be the set of elements in Example \ref{example} and in Example \ref{example2} respectively.
	Suppose that $k<2t+1$. Then we have to prove that $|R_{3.3}|>|R_{3.1}|$. Suppose to the contrary that $|R_{3.3}|\leq |R_{3.1}|$. Then 
	\begin{align*}
	\qbin{n-t-2}{k-t-2} \left( 1+ \theta_{t+1}q^{k-t-1}\frac{q^{n-k}-1}{q^{k-t-1}-1}  \right) &\leq \theta_{k}+\sum_{j=0}^{k-t-2} \qbin{k-t+1}{j+1}q^{(k-t-j)(k-t-j-1)}\qbin{n-k-1}{k-t-j-1}\\ & <q^{k-t+1}\theta_t+ \theta_{k-t}\qbin{n-t-1}{k-t-1}\\
	\end{align*}
	The last inequality follows since $q^{k-t+1}\theta_t-1$ is the number of elements of $R_{3.1}$ contained in $\langle \pi, \delta \rangle$ but not containing $\delta$. The second term $\theta_{k-t}\qbin{n-t-1}{k-t-1}$ is the number of $k$-spaces through a $(t+1)$-space in $\langle \pi, \delta \rangle$ through $\delta$.
	\begin{align*}
	&\xRightarrow{L. \ref{bounds}} \left(1+\frac 1q\right)q^{(n-k)(k-t-2)}\left( \theta_{t+1}q^{k-t-1}\frac{q^{n-k}-1}{q^{k-t-1}-1}  \right) < q^{k-t+1} \theta_t+2\theta_{k-t} q^{(n-k)(k-t-1)}\\
	& \Rightarrow \left(1+\frac 1q \right) (q^{t+2}-1)(q^{n-k}-1) q^{k-t-1} < \frac{ (q^{t+1}-1)(q^{k-t-1}-1)}{q^{(n-k)(k-t-2)-k+t-1}}+2(q^{k-t+1}-1)(q^{k-t-1}-1)q^{n-k}\\
	&\Rightarrow q^{n+1}+q^{k-t-1}+q^{n}+q^{k-t-2}-q^{n-t-1}-q^{k+1}-q^{n-t-2}-q^{k}\\ & \qquad <2q^{n+k-2t}+2q^{n-k}-2q^{n-t-1}-2q^{n-t+1}+\frac{q^{k-t+1} (q^{t+1}-1)(q^{k-t-1}-1)}{q^{(n-k)(k-t-2)}}\\
	& \Rightarrow\left( q^{n+1}-2q^{n+k-2t}-{q^{k-t+1-(n-k)(k-t-2)} (q^{t+1}-1)(q^{k-t-1}-1)}\right) +\left( q^{n}-q^{n-k}-q^{k+1} -q^k\right) \\ &\qquad + \left(  2q^{n-t-1}-q^{n-t-1}-q^{n-t-2}\right) +\left(2q^{n-t+1}+q^{k-t-1}+q^{k-t-2}-q^{n-k}\right)<0
	\end{align*}
	Now, the contradiction follows since all terms in the left hand side of the last inequality are positive. For the last three terms this follows immediately since $n>2k-t,k< 2t+1, k>t+1, q\geq 3$.  We end this proof by proving that the first term is also positive. Since $k\geq t+2$ and $n>2k-t=k+(k-t)\geq k+2$ we have that
	\begin{align*}
	&1<(n-k-1)(k-t-1)\\
	&\Leftrightarrow  n+1 > 2k-t+1-(n-k)(k-t-2)\\
	&\Rightarrow	q^{n+1}>q^{2k-t+1-(n-k)(k-t-2)}>q^{k-t+1-(n-k)(k-t-2)} (q^{t+1}-1)(q^{k-t-1}-1).
	\end{align*}
\end{proof}

\begin{lemma}\label{appendixaffienverschil3}
	Let $n>2k-t, q\geq 3$, and consider Example \ref{example} and Example \ref{example2} in $\AG(n,q)$. The number of elements in Example \ref{example2} is at least  the number of elements in Example \ref{example} if $k=2t+1$.
\end{lemma}
\begin{proof}
	Let $R_{3.1}$ and $R_{3.3}$ be the set of elements in Example \ref{example} and in Example \ref{example2} respectively.
	Suppose that $k= 2t+1$, then we have to prove that $|R_{3.3}|\geq |R_{3.1}|$.
	Consider the set $R_{3.3}$, and let $\tau$ be a $(t-1)$-space in $\AG(n,q)$, disjoint from $\Gamma$. Then we find that $|R_{3.3}|=\qbin{n-t-2}{t-1}+\sum_{i=-1}^{t-1} R_i(n,2t+1,t)$, with $R_i(n,k,t)=\{ \alpha\in R_{3.3}| \alpha\cap \Gamma \in \mathcal{R}, \dim(\tilde{\alpha}\cap \tilde{\tau})=i  \}$. 
	Since there are $\theta_{t+1}$ affine $(t+1)$-spaces in $\mathcal{R}$, and $\qbin{t}{i+1}$ projective $i$-spaces in $\tau$ we have, by using Lemma \ref{lemmadisjunct}, that
	\begin{align}
	|R_{3.3}|&=\qbin{n-t-2}{t-1}+\sum_{i=-1}^{t-1} \theta_{t+1} \qbin{t}{i+1}\left(q^{(t-i-1)^2}\qbin{n-2t-1}{t-i-1}-q^{(t-i-1)(t-i-2)}\qbin{n-2t-2}{t-i-2}\right)\nonumber \\
	&=\qbin{n-t-2}{t-1}+\sum_{i=-1}^{t-2} \theta_{t+1} \qbin{t}{i+1}q^{(t-i-1)(t-i-2)}\qbin{n-2t-2}{t-i-2}\frac{q^{n-t-i-2}-2q^{t-i-1}+1}{q^{t-i-1}-1}+\theta_{t+1}\nonumber\\
	&=\qbin{n-t-2}{t-1}+\sum_{j=0}^{t-1} \theta_{t+1} \qbin{t}{j}q^{(t-j)(t-j-1)}\qbin{n-2t-2}{t-j-1}\frac{q^{n-t-j-1}-2q^{t-j}+1}{q^{t-j}-1}+\theta_{t+1}.
	\end{align}
	On the other hand we find have that 
	\begin{align}\label{vgl3}
	|R_{3.1}|&=\theta_{2t+1}+\sum_{j=0}^{t-1} \qbin{t+2}{j+1}q^{(t+1-j)(t-j)}\qbin{n-2t-2}{t-j}.
	\end{align}
	Hence, it follows that
	\begin{align}\label{vgl4}
	|R_{3.3}|-|R_{3.1}| =\underbrace{\qbin{n-t-2}{t-1}+\theta_{t+1}-\theta_{2t+1}}_{=w_1} \nonumber+\sum_{j=0}^{t-1} q^{(t-j)(t-j-1)}\qbin{n-2t-2}{t-j}\qbin{t}{j }(q^{t+2}-1) w_2,
	\end{align}
	with
	\begin{align*}
		w_2=\frac{q^{n-t-j-1}-2q^{t-j}+1}{(q-1)(q^{n-3t+j-1}-1)} -\frac{q^{t+1}-1}{(q^{j+1}-1)(q^{t-j+1}-1)}q^{2(t-j)}.
	\end{align*}
	We will prove that $w_1\geq 0$ and $w_2\geq 0$, which proves that $|R_{3.3}|\geq |R_{3.1}|$ for $k= 2t+1$.
	\begin{align*}
	w_1=\qbin{n-t-2}{t-1}+\theta_{t+1}-\theta_{2t+1}&\overset{L. \ref{bounds}}{ \geq} \left(1+\frac{1}{q}\right) q^{(n-2t-1)(t-1)}+\theta_{t+1}-\frac{q^{2t+2}}{q-1}\\
	&=\frac{1}{q(q-1)}\left(  q^{(n-2t-1)(t-1)+2}-q^{(n-2t-1)(t-1)} -q^{2t+3} \right)+\theta_{t+1}.
	\end{align*}
	
	It is sufficient to prove that  $q^{(n-2t-1)(t-1)+2}>q^{2t+3} $. This inequality holds for $n> 2t+\frac{3t}{t-1}$. For $t>2$ this assumption holds since $n>2k-t=3t+2$.
	For $t=2$, the assumption holds for $n>10$. For $t=2,n=10$, we have that 
		$w_1=\theta_5+\theta_3-\theta_5>0$.

	Since $n>2k-t=3t+2$, the only remaining cases,  are $t=2$ and $n=9$, and $t=1$ and $n>5$. In these cases, we immediately calculate  $|R_{3.3}|- |R_{3.1}|$. For $t=2, n=9$, we have that $|R_{3.3}|-|R_{3.1}|=q^9+2q^8+3q^7+2q^6+q^5>0$. For $t=1,  n>5,$ we have that $|R_{3.3}|=|R_{3.1}|=1+q\theta_2\theta_{n-4}$.
	Now we investigate $w_2$:

	
	\begin{align*}
	w_2&=\frac{q^{n-t-j-1}-2q^{t-j}+1}{(q-1)(q^{n-3t+j-1}-1)} -\frac{q^{t+1}-1}{(q^{j+1}-1)(q^{t-j+1}-1)}q^{2(t-j)}\\
	  &= \frac{(q^{j+1}-1)(q^{t-j+1}-1)(q^{n-t-j-1}-2q^{t-j}+1)-(q-1)(q^{n-3t+j-1}-1)(q^{t+1}-1)q^{2(t-j)}}{(q-1)(q^{n-3t+j-1}-1)(q^{j+1}-1)(q^{t-j+1}-1)} \\
	& =\frac{\left(q^{n-j}+q^{n-j-t}-q^{n-t}-q^{n-2j}\right)+\left(q^{3t-2j+2}-q^{3t-2j+1}-2q^{2t-j+2}\right)}{(q-1)(q^{n-3t+j-1}-1)(q^{j+1}-1)(q^{t-j+1}-1)}\\
& \qquad +\frac{\left(q^{t+2}+2q^{t+1}-q^{t-j+1}-2q^{t-j}-q^{j+1}\right)+q^{2t-2j+1}+q^{2t-2j}+1}{(q-1)(q^{n-3t+j-1}-1)(q^{j+1}-1)(q^{t-j+1}-1)}.
	\end{align*}
	As $0\leq j\leq t-1$ and $q\geq 3$ we find that all terms in the denominator are at least $0$, which proves that $w_2\geq 0$.
	Hence, we find that $|R_{3.3}|\geq |R_{3.1}|$.
	
\end{proof}

\begin{lemma}\label{lemmaongelijkheid2x} Suppose $n>2k+t+2, q\geq 3, k>t+1, t>0$, then
	\[(\theta_{k-t})^x\qbin{n-t-x}{k-t-x}\qbin{t+x+1}{t+1}<(\theta_{k-t})^2\qbin{n-t-2}{k-t-2}\qbin{t+3}{t+1} \] for all $x>2$.
\end{lemma}
\begin{proof}
	Suppose to the contrary that
	\begin{align*}
	&	(\theta_{k-t})^x\qbin{n-t-x}{k-t-x}\qbin{t+x+1}{t+1}\geq (\theta_{k-t})^2\qbin{n-t-2}{k-t-2}\qbin{t+3}{t+1}  \\
	&	\xRightarrow{L. \ref{bounds}} \frac{q^{(k-t+1)(x-2)}}{(q-1)^{x-2}}2^2 q^{(n-k)(k-t-x)+x(t+1)}> \left(1+ \frac 1q\right)^2 q^{(n-k)(k-t-2)+2(t+1)}\\
	& \Rightarrow 4 > \left(1+\frac 1q\right)^2 (q-1)^{x-2} q^{(x-2)(n-2k-2)}\\
		& \xRightarrow[x>2, t>0]{ n>2k+t+2}4 > \left(1+\frac 1q\right)^2 (q-1) q.
	\end{align*}
	The last inequality gives a contradiction for $q\geq 3$.
\end{proof}

\begin{lemma}\label{lemmaappendixlelijk}
	Suppose  {$ k>t+1$, $t>0$,} $q\geq 4$, and $n> 2k+t+2,$ (or $q=3$ and $n> 2k+t+3$),
	then
	\[(\theta_{k-t})^x\qbin{n-t-x}{k-t-x}\qbin{t+x+1}{t+1}< f_p(q,n,k,t) \] for all $x\geq2$.
\end{lemma}
\begin{proof}
	From Lemma \ref{lemmaongelijkheid2x} it follows that it is sufficient to prove the lemma for $x=2$. Hence we have to prove the following inequalities:
	\begin{align}
		&(\theta_{k-t})^2\qbin{n-t-2}{k-t-2}\qbin{t+3}{t+1} < \qbin{n-t-2}{k-t-2} \left( 1+ \theta_{t+2}q^{k-t-1}\frac{q^{n-k}-1}{q^{k-t-1}-1}  \right)  &\text{ for } k\leq 2t+2; \label{eqeq}\\
		&(\theta_{k-t})^2\qbin{n-t-2}{k-t-2}\qbin{t+3}{t+1} < \theta_{k+1}+\sum_{j=0}^{k-t-2} \qbin{k-t+1}{j+1}q^{(k-t-j)(k-t-j-1)}\qbin{n-k-1}{k-t-j-1}  &\text{ for } k> 2t+2. \label{eqeqeq}
	\end{align}
	We start by proving inequality (\ref{eqeq}). Suppose to the contrary that this inequality does not hold. Then we have that
	\begin{align*}
	&(\theta_{k-t})^2\qbin{t+3}{t+1} \geq 1+ \theta_{t+2}q^{k-t-1}\frac{q^{n-k}-1}{q^{k-t-1}-1}  \\
	&\xRightarrow{L. \ref{bounds}} \frac{q^{2k-2t+2}}{(q-1)^2}2q^{2t+2} > \frac{q^{t+3}-1}{q-1} q^{k-t-1} \frac{q^{n-k}-1}{q^{k-t-1}-1}\\
	&\xRightarrow{n>2k+t+2} 2q^{k+t+5}(q^{k-t-1}-1)> (q-1)(q^{t+3}-1)(q^{n-k}-1)\geq (q-1)(q^{t+3}-1)(q^{k+t+3}-1)\\
	& \Rightarrow 0> (q^{k+2t+7}-q^{k+2t+6}-2q^{2k+4})+(2q^{k+t+5}-q^{t+4}-q^{k+t+4}-1)+q^{t+3}+q^{k+t+3}+q.
	\end{align*}
All terms in the right hand side of the last inequality are non negative since $k\leq 2t+2$ and $q\geq 3$. Hence we have a contradiction which proves (\ref{eqeq}).

Now we prove inequality (\ref{eqeqeq}). Suppose again to the contrary that this inequality does not hold. Then we have that
\begin{align}
&(\theta_{k-t})^2\qbin{n-t-2}{k-t-2}\qbin{t+3}{t+1} \geq \theta_{k+1}+\sum_{j=0}^{k-t-2} \qbin{k-t+1}{j+1}q^{(k-t-j)(k-t-j-1)}\qbin{n-k-1}{k-t-j-1}\nonumber \\
&\xRightarrow{j=0} (\theta_{k-t})^2\qbin{n-t-2}{k-t-2}\qbin{t+3}{t+1} > \theta_{k+1}+ \theta_{k-t}q^{(k-t)(k-t-1)}\qbin{n-k-1}{k-t-1}\nonumber \\
&\Rightarrow \theta_{k-t}\qbin{n-t-2}{k-t-2}\qbin{t+3}{t+1} > q^{(k-t)(k-t-1)}\qbin{n-k-1}{k-t-1}\label{eqq}\\
&\xRightarrow{L. \ref{bounds}} \frac{q^{k-t+1}}{q-1}2q^{(n-k)(k-t-2)} 2q^{2t+2}>q^{(k-t)(k-t-1)} \left(1+\frac 1q \right) q^{(n-2k+t)(k-t-1)} \nonumber\\
&\Leftrightarrow 4>(q-1)\left(q+1\right) q^{n-2k-t-4}. \nonumber
\end{align}
For $n\geq 2k+t+4$ we have that the last inequality gives a contradiction for $q\geq 3$. 
For $n=2k+t+3$ we find a contradiction for $q\geq 5$. Hence, we still have to prove inequality (\ref{eqeqeq}) for $n=2k+t+3$ and $q=4$.

For $n=2k+t+3$ and $q=4$ it follows, by using Lemma \ref{bounds} in equation (\ref{eqq}), that
\begin{align*}
	&\frac{q^{k-t+1}}{q-1}\left(1+\frac 2q\right)q^{2t+2} \left(1+\frac 2q\right) q^{(n-k)(k-t-2)}>q^{(k-t)(k-t-1)} \left(1+\frac 1q \right) q^{(n-2k+t)(k-t-1)} \\
	&\Leftrightarrow \left(q+2\right)^2>(q-1)\left(q+1\right)q.
\end{align*}
This gives a contradiction for $q=4$.
	
\end{proof}

\begin{lemma}\label{lemmaongelijkheid2xaffien} Suppose $n>2k+t+2, q\geq 3, k>t+1, t>0$, then
	\[(\theta_{k-t})^x\qbin{n-t-x}{k-t-x}q^x\qbin{t+x}{x}<(\theta_{k-t})^2\qbin{n-t-2}{k-t-2}q^2\qbin{t+2}{2} \] for all $x>2$.
\end{lemma}
\begin{proof}
	Suppose to the contrary that
	\begin{align*}
	&	(\theta_{k-t})^x\qbin{n-t-x}{k-t-x}q^x\qbin{t+x}{x}\geq (\theta_{k-t})^2\qbin{n-t-2}{k-t-2}q^2\qbin{t+2}{2}  \\
	&	\xRightarrow{L. \ref{bounds}} \frac{q^{(k-t+1)(x-2)}}{(q-1)^{x-2}}4 q^{(n-k)(k-t-x)+x+xt}> \left(1+ \frac 1q\right)^2 q^{(n-k)(k-t-2)+2+2t}\\
	& \Rightarrow 4 > \left(1+\frac 1q\right)^2 (q-1)^{x-2} q^{(x-2)(n-2k-2)}\\
	& \xRightarrow[x>2, t>0]{ n>2k+t+2}4 > \left(1+\frac 1q\right)^2 (q-1) q.
	\end{align*}
	The last inequality gives a contradiction for $q\geq 3$.
\end{proof}

\begin{lemma}\label{lemmaappendixaffienlelijk}
	Suppose Suppose  {$ k>t+1$, $t>0$,} $q\geq 4$, and $n> 2k+t+2,$ (or $q=3$ and $n> 2k+t+3$), then
	\[(\theta_{k-t})^x\qbin{n-t-x}{k-t-x}q^x\qbin{t+x}{x}< f_a(q,n,k,t) \] for all $x\geq2$.
\end{lemma}
\begin{proof}
	From Lemma \ref{lemmaongelijkheid2xaffien} it follows that it is sufficient to prove the lemma for $x=2$. Hence we have to prove the following inequalities:
	\begin{align}
	&(\theta_{k-t})^2\qbin{n-t-2}{k-t-2}q^2\qbin{t+2}{2} < \qbin{n-t-2}{k-t-2} \left( 1+ \theta_{t+1}q^{k-t-1}\frac{q^{n-k}-1}{q^{k-t-1}-1}  \right)  &\text{ for } k\leq 2t+1; \label{eqeqa}\\
	&(\theta_{k-t})^2\qbin{n-t-2}{k-t-2}q^2\qbin{t+2}{2} < \theta_{k}+\sum_{j=0}^{k-t-2} \qbin{k-t+1}{j+1}q^{(k-t-j)(k-t-j-1)}\qbin{n-k-1}{k-t-j-1}  &\text{ for } k> 2t+1. \label{eqeqeqa}
	\end{align}
	We start by proving inequality (\ref{eqeqa}). Suppose to the contrary that this inequality doesn't hold. Then we have that
	\begin{align*}
	&(\theta_{k-t})^2\qbin{t+2}{2}q^2 \geq 1+ \theta_{t+1}q^{k-t-1}\frac{q^{n-k}-1}{q^{k-t-1}-1}  \\
	&\xRightarrow{L. \ref{bounds}} \frac{q^{2k-2t+2}}{(q-1)^2}2q^{2t+2} > \frac{q^{t+2}-1}{q-1} q^{k-t-1} \frac{q^{n-k}-1}{q^{k-t-1}-1}\\
	&\xRightarrow{n>2k+t+2} 2q^{k+t+5}(q^{k-t-1}-1)> (q-1)(q^{t+2}-1)(q^{n-k}-1)>(q-1)(q^{t+2}-1)(q^{k+t+3}-1)\\
	& \Rightarrow 0> (q^{k+2t+6}-q^{k+2t+5}-2q^{2k+4})+(2q^{k+t+5}-q^{t+3}-q^{k+t+4}-1)+q^{t+2}+q^{k+t+3}+q.
	\end{align*}
	All terms in the right hand side of the last inequality are non negative since $k\leq 2t+1$ and $q\geq 3$. Hence we have a contradiction which proves (\ref{eqeqa}).
	
	Now we prove inequality (\ref{eqeqeqa}). Suppose again to the contrary that this inequality doesn't hold. Then we have that
	\begin{align}
	&(\theta_{k-t})^2\qbin{n-t-2}{k-t-2}q^2\qbin{t+2}{2} \geq \theta_{k}+\sum_{j=0}^{k-t-2} \qbin{k-t+1}{j+1}q^{(k-t-j)(k-t-j-1)}\qbin{n-k-1}{k-t-j-1}\nonumber \\
	&\xRightarrow{j=0} (\theta_{k-t})^2\qbin{n-t-2}{k-t-2}q^2\qbin{t+2}{2} > \theta_{k}+ \theta_{k-t}q^{(k-t)(k-t-1)}\qbin{n-k-1}{k-t-1}\nonumber \\
	&\Rightarrow \theta_{k-t}\qbin{n-t-2}{k-t-2}q^2\qbin{t+2}{2} > q^{(k-t)(k-t-1)}\qbin{n-k-1}{k-t-1}\label{eqqa}\\
	&\xRightarrow{L. \ref{bounds}} \frac{q^{k-t+1}}{q-1}2q^{(n-k)(k-t-2)}2q^{2t+2}>q^{(k-t)(k-t-1)} \left(1+\frac 1q \right) q^{(n-2k+t)(k-t-1)} \nonumber\\
	&\Leftrightarrow 4>(q-1)\left(q+1\right) q^{n-2k-t-4}. \nonumber
	\end{align}
	For $n\geq 2k+t+4$ we have that the last inequality gives a contradiction for $q\geq 3$. 
	For $n=2k+t+3$ we find a contradiction for $q\geq 5$. Hence, we still have to prove inequality (\ref{eqeqeqa}) for $n=2k+t+3$ and $q=4$.
	
	For $n=2k+t+3$ and $q=4$ there follows, by using Lemma \ref{bounds} in equation (\ref{eqqa}) that
	\begin{align*}
	&\frac{q^{k-t+1}}{q-1}\left(1+\frac 2q\right)q^{2t+2} \left(1+\frac 2q\right) q^{(n-k)(k-t-2)}>q^{(k-t)(k-t-1)} \left(1+\frac 1q \right) q^{(n-2k+t)(k-t-1)} \\
	&\Leftrightarrow \left(q+2\right)^2>(q-1)\left(q+1\right)q.
	\end{align*}
	This gives a contradiction for $q=4$.
	
\end{proof}

\begin{lemma}\label{appendixgeennaammeer}
		Suppose $n>2k+t+2, q\geq 3, k>t+1, t>0$, then
	\[	2\qbin{n-t-1}{k-t-1}+ (\theta_{t+1}\theta_{k-t}-\theta_{t+1}-1) \theta_{k-t} \qbin{n-t-2}{k-t-2}< f_p(q,n,k,t). \] 
\end{lemma}
\begin{proof}
	We have to prove the following inequalities:
	\begin{align}
	&2\qbin{n-t-1}{k-t-1}+ (\theta_{t+1}\theta_{k-t}-\theta_{t+1}-1) \theta_{k-t} \qbin{n-t-2}{k-t-2} \nonumber \\ &\hspace{2cm} < \qbin{n-t-2}{k-t-2} \left( 1+ \theta_{t+2}q^{k-t-1}\frac{q^{n-k}-1}{q^{k-t-1}-1}  \right)   \hspace{2,34cm}\text{ for } k\leq 2t+2; \label{eeq}\\
	&2\qbin{n-t-1}{k-t-1}+ (\theta_{t+1}\theta_{k-t}-\theta_{t+1}-1) \theta_{k-t} \qbin{n-t-2}{k-t-2}\nonumber \\&\hspace{2cm} < \theta_{k+1}+\sum_{j=0}^{k-t-2} \qbin{k-t+1}{j+1}q^{(k-t-j)(k-t-j-1)}\qbin{n-k-1}{k-t-j-1}   \text{ for } k> 2t+2. \label{eeqeq}
	\end{align}
	We start by proving inequality (\ref{eeq}). Suppose to the contrary that this inequality doesn't hold. Then we have that 
	\begin{align*}
	&2\qbin{n-t-1}{k-t-1}+ (\theta_{t+1}\theta_{k-t}-\theta_{t+1}-1) \theta_{k-t} \qbin{n-t-2}{k-t-2} \geq \qbin{n-t-2}{k-t-2} \left( 1+ \theta_{t+2}q^{k-t-1}\frac{q^{n-k}-1}{q^{k-t-1}-1}  \right)  \\
	&\Leftrightarrow 2 \frac{q^{n-t-1}-1}{q^{k-t-1}-1}+(\theta_{t+1}\theta_{k-t}-\theta_{t+1}-1) \theta_{k-t} \geq 1+ \theta_{t+2}q^{k-t-1}\frac{q^{n-k}-1}{q^{k-t-1}-1}\\
	&\Rightarrow 2 (q^{n-t-1}-1)+(q^{k-t-1}-1)(\theta_{t+1}\theta_{k-t}-\theta_{t+1}-1) \theta_{k-t} \geq \theta_{t+2}q^{k-t-1}(q^{n-k}-1)\\
	&\xRightarrow{L. \ref{bounds}}  2 (q^{n-t-1}-1)(q-1)+(q^{k-t-1}-1)(q^{k-t+1}-1)\frac{q^{k+3}}{(q-1)^2} > (q^{t+3}-1)q^{k-t-1}(q^{n-k}-1)\\
	& \Rightarrow 2q^{n-t}-2q^{n-t-1}-2q+2+ \frac{q^{3k-2t+3}}{(q-1)^2}> q^{n+2}-q^{n-t-1}-q^{k+2}+q^{k-t-1}\\
	& \Rightarrow 0> \left(q^{n+2}- \frac{q^{3k-2t+3}}{(q-1)^2}-2q^{n-t}\right)+\left(q^{n-t-1}-q^{k+2}-2\right)+q^{k-t-1}+2q\\
		& \xRightarrow[k\leq 2t+2]{n\geq 2k+t+3} 0> \left(q^{2k+3}(q^{t+2}-2)- \frac{q^{2k+5}}{(q-1)^2}\right)+\left(q^{2k+2}-q^{k+2}-2\right)+q^{k-t-1}+2q.
	\end{align*}
	For $q\geq 3$, we have that all terms on the right hand side of the last inequality are non negative. Hence we find a contradiction, which proves (\ref{eeq}).
	
	Now we prove inequality (\ref{eeqeq}) for $k>2t+2$. Suppose again to the contrary that this inequality doesn't hold. Then we have that 
	\begin{align*}
	&2\qbin{n-t-1}{k-t-1}+ (\theta_{t+1}\theta_{k-t}-\theta_{t+1}-1) \theta_{k-t} \qbin{n-t-2}{k-t-2} \\ & \hspace{4cm}\geq  \theta_{k+1}+\sum_{j=0}^{k-t-2} \qbin{k-t+1}{j+1}q^{(k-t-j)(k-t-j-1)}\qbin{n-k-1}{k-t-j-1}  \\
	&\xRightarrow{j=0} 2\qbin{n-t-1}{k-t-1}+ (\theta_{t+1}\theta_{k-t}-\theta_{t+1}-1) \theta_{k-t} \qbin{n-t-2}{k-t-2} \geq  \theta_{k-t}q^{(k-t)(k-t-1)}\qbin{n-k-1}{k-t-1} \\
	&\xRightarrow{L. \ref{bounds}} 4q^{(n-k)(k-t-1)}+ \frac{q^{2k-t+4}}{(q-1)^3} 2 q^{(n-k)(k-t-2)} \geq  \theta_{k-t}q^{(k-t)(k-t-1)}\left( 1+\frac 1q \right) q^{(n-2k+t)(k-t-1)}  \\
	&\Rightarrow  4+ \frac{2}{q^{n-3k+t-4}(q-1)^3}   \geq  \theta_{k-t}\left( 1+\frac 1q \right) >\theta_{k-t}+4\\
	&\xRightarrow{n\geq 2k+t+3} \frac{2q^{k-2t+1}}{(q-1)^2}   >  q^{k-t+1}-1\\
	&\xRightarrow{q\geq 3} q^{k-2t+1} >  q^{k-t+1}-1.
	\end{align*}
	The last inequality gives a contradiction for $q\geq 3, t>0$.
\end{proof}

\begin{lemma}\label{appendixaffienbla}
	Suppose $n>2k+t+2, q\geq 3, k>t+1, t>0$, then
	\[2\qbin{n-t-1}{k-t-1}+ (\theta_{t+1}\theta_{k-t}-\theta_{t+1}-\theta_{k-t}) \theta_{k-t} \qbin{n-t-2}{k-t-2}< f_a(q,n,k,t). \] 
	\end{lemma}
\begin{proof}
	We have to prove the following inequalities:
	\begin{align}
	2\qbin{n-t-1}{k-t-1}&+ (\theta_{t+1}\theta_{k-t}-\theta_{t+1}-\theta_{k-t}) \theta_{k-t} \qbin{n-t-2}{k-t-2} \nonumber \\ &< \qbin{n-t-2}{k-t-2} \left( 1+ \theta_{t+1}q^{k-t-1}\frac{q^{n-k}-1}{q^{k-t-1}-1}  \right)   \hspace{2cm}\text{ for } k\leq 2t+1; \label{eeeq}\\
	2\qbin{n-t-1}{k-t-1}&+ (\theta_{t+1}\theta_{k-t}-\theta_{t+1}-\theta_{k-t}) \theta_{k-t} \qbin{n-t-2}{k-t-2}\nonumber \\&< \theta_{k}+\sum_{j=0}^{k-t-2} \qbin{k-t+1}{j+1}q^{(k-t-j)(k-t-j-1)}\qbin{n-k-1}{k-t-j-1}   \text{ for } k> 2t+1. \label{eeeqeq}
	\end{align}
	We start by proving inequality (\ref{eeeq}). Suppose to the contrary that this inequality doesn't hold. Then we have that 
	\begin{align*}
	&2\qbin{n-t-1}{k-t-1}+ (\theta_{t+1}\theta_{k-t}-\theta_{t+1}-\theta_{k-t}) \theta_{k-t} \qbin{n-t-2}{k-t-2} \geq \qbin{n-t-2}{k-t-2} \left( 1+ \theta_{t+1}q^{k-t-1}\frac{q^{n-k}-1}{q^{k-t-1}-1}  \right)  \\
	&\Leftrightarrow 2 \frac{q^{n-t-1}-1}{q^{k-t-1}-1}+(\theta_{t+1}\theta_{k-t}-\theta_{t+1}-\theta_{k-t}) \theta_{k-t} \geq 1+ \theta_{t+1}q^{k-t-1}\frac{q^{n-k}-1}{q^{k-t-1}-1}\\
	&\Rightarrow 2 (q^{n-t-1}-1)+(q^{k-t-1}-1)(\theta_{t+1}\theta_{k-t}-\theta_{t+1}-\theta_{k-t}) \theta_{k-t} > \theta_{t+1}q^{k-t-1}(q^{n-k}-1)\\
	&\xRightarrow{L. \ref{bounds}}  2 (q^{n-t-1}-1)(q-1)+(q^{k-t-1}-1)(q^{k-t+1}-1)\frac{q^{k+3}}{(q-1)^2} > (q^{t+2}-1)q^{k-t-1}(q^{n-k}-1)\\
	& \Rightarrow 2q^{n-t}-2q^{n-t-1}-2q+2+ \frac{q^{3k-2t+3}}{(q-1)^2}> q^{n+1}-q^{n-t-1}-q^{k+1}+q^{k-t-1}\\
	& \Rightarrow 0> \left(q^{n+1}- \frac{q^{3k-2t+3}}{(q-1)^2}-2q^{n-t}\right)+\left(q^{n-t-1}-q^{k+1}-2\right)+q^{k-t-1}+2q\\
	& \xRightarrow[k\leq 2t+1]{n\geq 2k+t+3} 0> \left(q^{2k+3}(q^{t+1}-2)- \frac{q^{2k+4}}{(q-1)^2}\right)+\left(q^{2k+2}-q^{k+1}-2\right)+q^{k-t-1}+2q.
	\end{align*}
	For $q\geq 3$, we have that all terms on the right hand side of the last inequality are non negative. Hence we find a contradiction, which proves (\ref{eeeq}).
	
	Now we prove inequality (\ref{eeeqeq}). Suppose again to the contrary that this inequality doesn't hold. Then we have that 
	\begin{align*}
	&2\qbin{n-t-1}{k-t-1}+ (\theta_{t+1}\theta_{k-t}-\theta_{t+1}-\theta_{k-t} ) \theta_{k-t} \qbin{n-t-2}{k-t-2} \\&\hspace{4cm} \geq  \theta_{k}+\sum_{j=0}^{k-t-2} \qbin{k-t+1}{j+1}q^{(k-t-j)(k-t-j-1)}\qbin{n-k-1}{k-t-j-1}  \\
	&\xRightarrow{j=0} 2\qbin{n-t-1}{k-t-1}+ (\theta_{t+1}\theta_{k-t}-\theta_{t+1}-\theta_{k-t} ) \theta_{k-t} \qbin{n-t-2}{k-t-2} \geq  \theta_{k-t}q^{(k-t)(k-t-1)}\qbin{n-k-1}{k-t-1} \\
	&\xRightarrow{L. \ref{bounds}} 4q^{(n-k)(k-t-1)}+ \frac{q^{2k-t+4}}{(q-1)^3} 2 q^{(n-k)(k-t-2)} \geq  \theta_{k-t}q^{(k-t)(k-t-1)}\left( 1+\frac 1q \right) q^{(n-2k+t)(k-t-1)}  \\
	&\Rightarrow  4+ \frac{2}{q^{n-3k+t-4}(q-1)^3}   \geq  \theta_{k-t}\left( 1+\frac 1q \right) >\theta_{k-t}+4\\
	&\xRightarrow{n\geq 2k+t+3} \frac{2q^{k-2t+1}}{(q-1)^2}   >  q^{k-t+1}-1\\
	&\xRightarrow{q\geq 3} q^{k-2t+1} >  q^{k-t+1}-1.
	\end{align*}
	The last inequality gives a contradiction for $q\geq 3$.
\end{proof}

\begin{lemma}\label{appendixlaatste}
	Suppose that $n>2k+t+2, k>2t+2$, $2\leq x\leq k-t+1, t>0$. Then we have that 
	\begin{align*}
		 &\theta_{k+1}+\sum_{j=0}^{k-t-2} \qbin{k-t+1}{j+1}q^{(k-t-j)(k-t-j-1)}\qbin{n-k-1}{k-t-j-1}\\& \hspace{3cm}>\theta_{t+x}\qbin{n-t-x+1}{k-t-x+1}+\theta_{k-t}^{2}\qbin{n-t-2}{k-t-2}+\theta_{k-t-1}\qbin{n-t-1}{k-t-1}.
	\end{align*}
	
\end{lemma}
\begin{proof}
	We first prove that $\theta_{t+x+1}\qbin{n-t-x}{k-t-x}<\theta_{t+x}\qbin{n-t-x+1}{k-t-x+1}$ for $2 \leq x< k-t+1$.
	\begin{align*}
	&\theta_{t+x+1}\qbin{n-t-x}{k-t-x}<\theta_{t+x}\qbin{n-t-x+1}{k-t-x+1}\\
	&\Leftrightarrow q^{t+x+2}-1<(q^{t+x+1}-1)\frac{q^{n-t-x+1}-1}{q^{k-t-x+1}-1}\\
	&\Leftrightarrow q^{k+3}-q^{t+x+2}-q^{k-t-x+1}<q^{n+2}-q^{n-t-x+1}-q^{t+x+1}\\
	&\Leftrightarrow -q^{t+x+2}-q^{k-t-x+1}<q^{n+2}-q^{n-t-x+1}-q^{t+x+1}-q^{k+3}.
	\end{align*}
	Note that the right hand side of the last inequality is positive for $q\geq 3$, which proves the inequality.

	Suppose now to the contrary that the inequality in the statement of the lemma doesn't hold. Then we have from the previous observation that 
	\begin{align*}
		&\theta_{k+1}+\sum_{j=0}^{k-t-2} \qbin{k-t+1}{j+1}q^{(k-t-j)(k-t-j-1)}\qbin{n-k-1}{k-t-j-1}\\ & \hspace{2.5cm}\leq\theta_{t+x}\qbin{n-t-x+1}{k-t-x+1}+\theta_{k-t}^{2}\qbin{n-t-2}{k-t-2}+\theta_{k-t-1}\qbin{n-t-1}{k-t-1}\\
		&\xRightarrow{j=0} \theta_{k-t}q^{(k-t)(k-t-1)}\qbin{n-k-1}{k-t-1}  < \theta_{t+2}\qbin{n-t-1}{k-t-1}+\theta_{k-t}^{2}\qbin{n-t-2}{k-t-2}+\theta_{k-t-1}\qbin{n-t-1}{k-t-1}\\
		&\xRightarrow{L. \ref{bounds}} \frac{q^{k-t+1}-1}{q-1}q^{(k-t)(k-t-1)} \left(1+\frac 1q\right)q^{(n-2k+t)(k-t-1)}\\ & \hspace{2.5cm} <\frac{q^{t+3}-1}{q-1}2q^{(n-k)(k-t-1)}+\frac{(q^{k-t+1}-1)^2}{(q-1)^2}2q^{(n-k)(k-t-2)}+\frac{q^{k-t}-1}{q-1}2q^{(n-k)(k-t-1)}\\
		&\Rightarrow (q^{k-t+1}-1)\left(1+\frac 1q\right) <2(q^{t+3}-1)+2\frac{(q^{k-t+1}-1)^2}{q^{n-k}(q-1)}+2(q^{k-t}-1)\\
		&\Rightarrow \left(q^{k-t+1}-2q^{t+3}-2q^{k-t}\right)+\left(q^{k-t}+3-\frac 1q-2\frac{(q^{k-t+1}-1)^2}{q^{n-k}(q-1)}\right)<0
	\end{align*}
	For $k>2t+3, q\geq 3$ and $n>2k+t+2$ both terms in the left hand side of the last inequality are non negative, which gives a contradiction.
	For $k=2t+3$ we have 
	\begin{align*}
		\left(q^{t+4}-3q^{t+3}\right)+\left(3-\frac 1q-2\frac{(q^{t+4}-1)^2}{q^{n-2t-3}(q-1)}\right)<0\\
		\xRightarrow{n\geq 5t+9}\left(q^{t+4}-3q^{t+3}\right)+\left(3-\frac 1q-2\frac{(q^{t+4}-1)^2}{q^{3t+7}(q-1)}\right)<0,\\
	\end{align*}
	which also gives a contradiction for $q\geq 3$ and $t>0$.
\end{proof}

\begin{lemma}\label{appendixlaatsteaffienextra}
	Suppose that $n>2k+t+2, k>2t+1$, $3\leq x\leq k-t+1, t>0, q\geq 3$. Then we have that  
	\begin{align*}
		&\theta_{k}+\sum_{j=0}^{k-t-2} \qbin{k-t+1}{j+1}q^{(k-t-j)(k-t-j-1)}\qbin{n-k-1}{k-t-j-1}\\ &\hspace{3cm}>\theta_{t+x}\qbin{n-t-x+1}{k-t-x+1}+\theta_{k-t}^{2}\qbin{n-t-2}{k-t-2}+\theta_{k-t-1}\qbin{n-t-1}{k-t-1}.
	\end{align*}

\end{lemma}
\begin{proof}
	From the proof of Lemma \ref{appendixlaatste}, we know that  $\theta_{t+x+1}\qbin{n-t-x}{k-t-x}<\theta_{t+x}\qbin{n-t-x+1}{k-t-x+1}$ for $3 \leq x< k-t+1$.
	
	Suppose now to the contrary that the inequality in the statement of the lemma doesn't hold. Then we have from the previous observation that 
	\begin{align*}
	&\theta_{k}+\sum_{j=0}^{k-t-2} \qbin{k-t+1}{j+1}q^{(k-t-j)(k-t-j-1)}\qbin{n-k-1}{k-t-j-1}\\ & \hspace{2.5cm}\leq\theta_{t+x}\qbin{n-t-x+1}{k-t-x+1}+\theta_{k-t}^{2}\qbin{n-t-2}{k-t-2}+\theta_{k-t-1}\qbin{n-t-1}{k-t-1}\\
	&\xRightarrow[x\geq 3]{j=0} \theta_{k-t}q^{(k-t)(k-t-1)}\qbin{n-k-1}{k-t-1}  < \theta_{t+3}\qbin{n-t-2}{k-t-2}+\theta_{k-t}^{2}\qbin{n-t-2}{k-t-2}+\theta_{k-t-1}\qbin{n-t-1}{k-t-1}\\
	&\xRightarrow{L. \ref{bounds}} \frac{q^{k-t+1}-1}{q-1}q^{(k-t)(k-t-1)} \left(1+\frac 1q\right)q^{(n-2k+t)(k-t-1)}\\ & \hspace{2.5cm} <\frac{q^{t+4}-1}{q-1}2q^{(n-k)(k-t-2)}+\frac{(q^{k-t+1}-1)^2}{(q-1)^2}2q^{(n-k)(k-t-2)}+\frac{q^{k-t}-1}{q-1}2q^{(n-k)(k-t-1)}\\
	&\Rightarrow (q^{k-t+1}-1)\left(1+\frac 1q\right) <2\frac{q^{t+4}-1}{q^{n-k}}+2\frac{(q^{k-t+1}-1)^2}{q^{n-k}(q-1)}+2(q^{k-t}-1)\\
	&\Rightarrow \left(q^{k-t+1}-2\frac{q^{t+4}-1}{q^{n-k}}-2q^{k-t}\right)+\left(q^{k-t}+1-\frac 1q-2\frac{(q^{k-t+1}-1)^2}{q^{n-k}(q-1)}\right)<0
	\end{align*}
	For $n>2k+t+2, q\geq 3$  both terms in the left hand side of the last inequality are non negative, which gives a contradiction.
\end{proof}

\begin{lemma}\label{appendixlaatsteaffien}
	Suppose that $n>2k+t+2, k>2t+1$ and $q\geq 3$. Then we have that 
	\begin{align*}
		 &\theta_{k}+\sum_{j=0}^{k-t-2} \qbin{k-t+1}{j+1}q^{(k-t-j)(k-t-j-1)}\qbin{n-k-1}{k-t-j-1}\\& \hspace{3cm}>q^2 \theta_{t-1}\qbin{n-t-1}{k-t-1}+\theta_{k-t}^{2}\qbin{n-t-2}{k-t-2}+\theta_{k-t-1}\qbin{n-t-1}{k-t-1}
	\end{align*}
	
\end{lemma}
\begin{proof}
	Suppose to the contrary that the inequality in the statement of the lemma doesn't hold. Then we have from the previous observation that 
	\begin{align*}
	&\theta_{k}+\sum_{j=0}^{k-t-2} \qbin{k-t+1}{j+1}q^{(k-t-j)(k-t-j-1)}\qbin{n-k-1}{k-t-j-1}\\ & \hspace{2.5cm}\leq q^2\theta_{t-1}\qbin{n-t-1}{k-t-1}+\theta_{k-t}^{2}\qbin{n-t-2}{k-t-2}+\theta_{k-t-1}\qbin{n-t-1}{k-t-1}\\
	&\xRightarrow{j=0} \theta_{k} +\theta_{k-t}q^{(k-t)(k-t-1)}\qbin{n-k-1}{k-t-1}  < q^2\theta_{t-1}\qbin{n-t-1}{k-t-1}+\theta_{k-t}^{2}\qbin{n-t-2}{k-t-2}+\theta_{k-t-1}\qbin{n-t-1}{k-t-1}\\
	&\xRightarrow{L. \ref{bounds}} \frac{q^{k-t+1}-1}{q-1}q^{(k-t)(k-t-1)} \left(1+\frac 1q\right)q^{(n-2k+t)(k-t-1)}\\ & \hspace{2.5cm} <\frac{q^{t}-1}{q-1}2q^{(n-k)(k-t-1)+2}+\frac{(q^{k-t+1}-1)^2}{(q-1)^2}2q^{(n-k)(k-t-2)}+\frac{q^{k-t}-1}{q-1}2q^{(n-k)(k-t-1)}\\
	&\xRightarrow{} (q^{k-t+1}-1)\left(1+\frac 1q\right) <2(q^{t+2}-q^2)+2\frac{(q^{k-t+1}-1)^2}{q^{n-k}(q-1)}+2(q^{k-t}-1)\\
	&\Rightarrow \left(q^{k-t+1}-2q^{t+2}-2q^{k-t}\right)+\left(q^{k-t}+1+2q^2-\frac 1q-2\frac{(q^{k-t+1}-1)^2}{q^{n-k}(q-1)}\right)<0
	\end{align*}
	For $k>2t+2, q\geq 3$ and $n>2k+t+2$ both terms in the left hand side of the last inequality are non negative, which gives a contradiction.
	For $k=2t+2$ we have that $n>2k+t+2=5t+6$, and so
	\begin{align*}
	\left(q^{t+3}-3q^{t+2}\right)+\left(2q^2+1-\frac 1q-2\frac{(q^{t+3}-1)^2}{q^{n-2t-2}(q-1)}\right)<0\\
	\xRightarrow{n\geq 5t+7}\left(q^{t+3}-3q^{t+2}\right)+\left(2q^2+1-\frac 1q-2\frac{(q^{t+3}-1)^2}{q^{3t+5}(q-1)}\right)<0,\\
	\end{align*}
	which also gives a contradiction for $q\geq 3$.
\end{proof}

\subsection*{Acknowledgements}
The research of Jozefien D'haeseleer is supported by the FWO (Research Foundation Flanders).


\begin{thebibliography}{1}
	\bibitem{q-analogue_hilton} Blokhuis A., Brouwer  A.E.,  Chowdhury A., Frankl P., Patkòs B., Mussche T., Sz{\H{o}}nyi T., A Hilton-Milner theorem for vector spaces.
	\textit{Electron. J. Combin.}, 17(1), 2010. 
\bibitem{brouwer}
Brouwer A.E., Cohen A.M., Neumaier A., \textit{Distance-regular graphs}, Vol. 18 of Ergeb. Math. Grenzgeb. (3). Springer-Verlag, 1989. 

\bibitem{nieuw} Cao M., Lv B., Wang K., Zhou S., Non-trivial $t$-intersecting families for vector spaces. arXiv:2007.11767.
	
	\bibitem{Gong} Chao G., Benjian L., Kaishun W., Non-trivial intersecting families for finite
	affine spaces.  arXiv:1901.05759.
	
		\bibitem{Maarten} De Boeck M., The largest Erd{\H{o}}s-Ko-Rado sets of planes in finite projective and finite classical polar spaces. \textit{Des. Codes Cryptogr.} 72(1): 77–11, 2014.
	
	\bibitem{surveyEKR} De Boeck M., Storme L.,
	Theorems of Erd{\H{o}}s-Ko-Rado type in geometrical settings.
	\textit{Science China Mathematics} 56(7): 1333-1348,
	2013.
\bibitem{ellis} Ellis D., Nontrivial $t$-intersecting families of subspaces. \textit{private communication}.

\bibitem{erdos_ko_rado} Erd{\H{o}}s P., Ko C., Rado R. Intersection theorems for systems of finite sets. \textit{Quart. J. Math. Oxford Ser.} 12(2):
313–320, 1961.

\bibitem{frankl} Frankl P., Wilson R.M., The Erd{\H{o}}s-Ko-Rado theorem for vector spaces. \textit{J. Combin. Theory Ser. A} 43(2):228-236, 1986.

\bibitem{EKRpermutation}  Godsil C., Meagher K., A new proof of the Erd{\H{o}}s-Ko-Rado theorem for intersecting families of permutations, \textit{European J. Combin.} 30 (2):404–414, 2009.

\bibitem{GuoXu} Guo J., Xu Q., The Erd\H{o}s-Ko-Rado theorem for finite affine spaces. \textit{Linear Multilinear Algebra} 65(3), 593-599, 2017.

\bibitem{hilton_minler} Hilton J.W., Milner E.C., Some intersection theorems for systems
of finite sets. \textit{Quart. J. Math. Oxford Ser. 18(2)}, 1967.

\bibitem{thesisferdinand} Ihringer F., Finite Geometry intersecting algebraic combinatorics. PhD thesis, Justus-Liebig-Universit{\"a}t Gie{\ss}en, 2015. 

\bibitem{EKRdesign} Rands B.M.I., An extension of the Erd{\H{o}}s-Ko-Rado theorem to t-designs, \textit{J. Combin. Theory Ser. A}, 32(3):391–395, 1982.

\bibitem{Segre} Segre B., \textit{Lectures on Modern Geometry (with an appendix by L. Lombardo-Radice).} Consiglio Nazionale delle Richerche, Monografie Matematiche. Edizioni Cremonese, Roma, 1961.

\bibitem{wilsonEKR} Wilson R. M., The exact bound in the Erd{\H{o}}s-Ko-Rado theorem. \textit{Combinatorica} 4: 247–257, 1984.
\end{thebibliography}
\end{document}